\documentclass[11pt,twoside, a4paper, english, reqno]{amsart}
\usepackage{amssymb,amsfonts,latexsym,color}
\usepackage{graphics,verbatim}
\usepackage{graphicx}
\newtheorem{theorem}{Theorem}[section]
\newtheorem{proposition}{Proposition}[section]
\newtheorem{lemma}{Lemma}[section]

\newtheorem{remark}{Remark}[section]
\newtheorem{definition}{Definition}[section]

 \newcommand{\<}{\left\langle}
\renewcommand{\>}{\right\rangle}

\newcommand{\eps}{\varepsilon}

\newcommand{\abs}[1]{\left\vert#1\right\vert}

\newcommand{\norm}[1]{\left\Vert#1\right\Vert}

\newcommand{\be} {\begin{equation}}
\newcommand{\ee} {\end{equation}}
\newcommand{\bea} {\begin{eqnarray}}
\newcommand{\eea} {\end{eqnarray}}
\newcommand{\Bea} {\begin{eqnarray*}}
\newcommand{\Eea} {\end{eqnarray*}}
\newcommand{\pa} {\partial}

\newcommand{\al} {\alpha}
\newcommand{\ba} {\beta}
\newcommand{\de} {\delta}

\newcommand{\ga} {\gamma}

\newcommand{\Om} {\Omega}

\newcommand{\De} {\Delta}
\newcommand{\la} {\lambda}

\newcommand{\nequiv} {\not\equiv}

\newcommand{\no} {\nonumber}
\newcommand{\noi} {\noindent}
\newcommand{\lab} {\label}

\newcommand{\var} {\varepsilon}
\newcommand{\f}{\frac}

\newcommand{\R}{\mathbb R}
\newcommand{\N}{\mathbb N}
\newcommand{\Rn}{\mathbb R^N}
\newcommand{\Iom}{\int_{\Omega}}
\newcommand{\deb}{\rightharpoonup}
\catcode`\@=11

\makeatletter \@addtoreset{equation}{section} \makeatother

\begin{document}

 \title[Sign changing solutions of p-fractional equations]{Sign changing solutions of p-fractional equations with concave-convex nonlinearities}
 \author{Mousomi Bhakta, \ Debangana Mukherjee}
 \address{Department of Mathematics, Indian Institute of Science Education and Research, Dr. Homi Bhaba Road, Pune-411008, India}
 \email{M. Bhakta: mousomi@iiserpune.ac.in, \  D. Mukherjee: debangana18@gmail.com}
 
 \maketitle
 \begin{abstract}
 In this article we study the existence of sign changing solution of the following p-fractional problem with concave-critical nonlinearities:
\begin{eqnarray*}
  (-\Delta)^s_pu &=& \mu |u|^{q-1}u + |u|^{p^*_s-2}u \quad\mbox{in}\quad \Omega,\\
  u&=&0\quad\mbox{in}\quad\mathbb{R}^N\setminus\Omega,
\end{eqnarray*} 
where $s\in(0,1)$ and $p\geq 2$ are fixed parameters, $0<q<p-1$, $\mu\in\mathbb{R}^+$ and $p_s^*=\frac{Np}{N-ps}$.  $\Omega$ is an open, bounded domain in $\mathbb{R}^N$ with smooth boundary with $N>ps$ .  
 \end{abstract}

\section{\bf Introduction}
Let us consider the fractional p-Laplace equation with concave-critical nonlinearities
\begin{align*}
 \left(\mathcal{P}_{\mu}\right)
 \begin{cases}
  (-\Delta)^s_pu = \mu |u|^{q-1}u + |u|^{p^*_s-2}u &\quad\mbox{in}\quad \Omega,\\
  u=0&\quad\mbox{in}\quad\mathbb{R}^N\setminus\Omega,
 \end{cases}
\end{align*}
where $s\in(0,1)$, $p>1$ are fixed, $N>ps$,  $\Omega$ is an open, bounded domain in $\Rn$ with smooth boundary,
$0<q<p-1$, $p^*=\frac{Np}{N-ps}$ and $\mu\in\R^+$. The non-local operator $(-\Delta)^s_p$ is defined as follows:
\begin{align} \label{frac s_p}
 (-\Delta)^s_pu(x)=2\lim_{\eps\to 0}\int_{\mathbb{R}^N\setminus B_\eps(x)}\frac{|u(y)-u(x)|^{p-2}(u(y)-u(x))}{|x-y|^{N+ps}}dy,\,\,\,x\in\mathbb{R}^N. 
\end{align}
For $ p\geq 1$, we denote the usual fractional Sobolev space by $W^{s,p}(\Omega)$ endowed with the norm
$$
||{u}||_{W^{s,p}(\Om)}:=||{u}||_{L^p(\Om)}+\left(\int_{\Om\times\Om} \frac{|u(x)-u(y)|^p}{|x-y|^{N+sp}}dxdy\right)^{1/p}.
$$
We set $Q:=\R^{2N}\setminus (\Om^c \times \Om^c)$ with $\Om^c=\Rn \setminus \Om$
and define 
$$
X:=\Big\{u:\mathbb{R}^N\to\mathbb{R}\mbox{ measurable }\Big|u|_{\Omega}\in L^p(\Omega)\mbox{ and }
   \int_{Q} \frac{|u(x)-u(y)|^p}{|x-y|^{N+sp}}dxdy<\infty\Big\}.
$$
 The space $X$ is endowed with the norm defined as
   $$||{u}||_X=||{u}||_{L^p(\Om)}+\left(\int_{Q} \frac{|u(x)-u(y)|^p}{|x-y|^{N+sp}}dxdy\right)^{1/p}.$$
Then, we define $X_0 :=\Big\{u \in X:u=0 \quad\text{a.e. in}\quad \Rn \setminus \Om\Big\} $ or equivalently 
 as $\overline{C_c^\infty(\Om)}^X$ and for any $p>1$, $X_0$ is a uniformly convex Banach space (see \cite{GS}) endowed with the norm   
   $$||{u}||_{X_0}=\left(\int_{Q} \frac{|u(x)-u(y)|^p}{|x-y|^{N+sp}}dxdy\right)^{1/p}.$$
 Since $u=0$ in $\Rn\setminus\Om,$ the above integral can be extended to all of $\mathbb{R}^N.$ The embedding
 $X_0\hookrightarrow L^r(\Om)$ is continuous for any $r\in[1,p^*_s]$ and compact for $r\in[1,p^*_s).$  For further details on $X_0$ and it's properties we refer \cite{NePaVal}.

   \begin{definition}\label{def-sol} We say that $u\in X_0$ is a weak solution of 
   $(\mathcal{P}_\mu)$ if 
   \begin{eqnarray*} 
 \int_{\R^{2N}}\frac{|u(x)-u(y)|^{p-2}(u(x)-u(y))(\phi(x)-\phi(y))}{|x-y|^{N+ps}}dxdy &=& \mu\Iom |u(x)|^{q-1}u(x)\phi(x)dx \\
 &+& \Iom |u(x)|^{p^*_s-2}u(x)\phi(x)dx,
 \end{eqnarray*}
 for all $\phi \in X_0. $
 \end{definition}
 The Euler-Lagrange energy functional associated to  $(\mathcal{P}_\mu$) is
 \bea\lab{I-mu-la}
 I_\mu(u) &=& \frac{1}{p}\int_{\R^{2N}}\frac{|u(x)-u(y)|^p}{|x-y|^{N+ps}}dxdy 
 -\frac{\mu}{q+1}\Iom|u|^{q+1}dx-\frac{1}{p^*_s}\Iom|u|^{p^*_s}dx\no \\
 &=& \frac{1}{p}\norm{u}_{X_0}^p
 -\frac{\mu}{q+1}|u|_{L^{q+1}(\Om)}^{q+1}-\frac{1}{p^*_s}|u|_{L^{p^*_s}(\Om)}^{p^*_s}.
 \eea
 We define the best fractional critical Sobolev constant $S$ as
\be\lab{S}
S := \inf_{v \in W^{s,p}(\mathbb{R}^N) \setminus \{0\}}
\frac{\displaystyle\int_{\R^{2N}}\frac{|v(x)-v(y)|^p}{|x-y|^{N+ps}}dxdy}
{\left(\displaystyle\int_{\mathbb{R}^N} |v(x)|^{p^*_s}dx\right)^{p/p^*_s}},
\ee
which is positive by fractional Sobolev inequality. Since the embedding $X_0\hookrightarrow L^{p^*_s}$ is not compact, $I_{\mu}$ does
not satisfy the Palais-Smale condition globally, but that holds true when the energy level falls inside a suitable range related to $S$. As it was mentioned in \cite{CS}, the main difficulty dealing with critical fractional case with $p\not=2$, is  the lack of an explicit formula for minimizers
of $S$ which is very often a key tool to handle the estimates leading to the compactness range of $I_{\mu}$. This difficulty has been tactfully overcome in \cite{CS} and \cite{MPSY}  by the optimal asymptotic
behavior of minimizers, which was recently obtained in \cite{BMS}. Using the same optimal asymptotic behavior of minimizer of $S$, we will establish suitable compactness range.

Thanks to the continuous Sobolev embedding $X_0\hookrightarrow L^{p^*_s}(\Rn)$, 
$I_\mu$ is well defined $C^1$ functional on $X_0$. It is well known that there exists a one-to-one correspondence
between the weak solutions of $(\mathcal{P}_\mu)$ and the critical points of $I_{\mu}$ on $X_0$. 

A classical topic in nonlinear analysis is the study of existence and multiplicity
of solutions for nonlinear equations. In past few years there has been considerable interest in studying the following general   fractional p-Laplacian problem
\Bea
(-\De)^s_p u &=&f(u) \quad\text{in}\quad\Om,\\
u &=&0 \quad\text{in}\quad\Rn\setminus\Om.
\Eea
In \cite{LL}, the eigenvalue problem associated with $(-\De)^s_p$ has been studied. Some results about the existence of solutions have been considered in \cite{ GS2, ILPS, LL}, see also the references therein.

On the other hand, the fractional problems for $p=2$ have been investigated
by many researchers, see for example \cite{SerVal1} for the subcritical case, \cite{BCSS, BM, SerVal} for the critical case. In \cite{BrCPS} the authors  studied the nonlocal equation involving a concave-convex nonlinearity in the
subcritical case. In \cite{CD} the existence of multiple positive solutions to $(\mathcal{P}_{\mu})$  for both the subcritical and critical case were obtained.  Existence of infinitely many nontrivial solution to $(\mathcal{P}_{\mu})$ in both subcritical and critical cases and  existence of at least one sign-changing solution have been established in \cite{BM}. In the local case $s=1$
equation with concave-convex nonlinearities were studied by many authors, to mention few, see \cite{ABC, AAP, BW, CCP}.  When $s=1$ and $p=2$, existence of sign changing solution was studied in \cite{Chen}.


In \cite{GS}, Goyal and Sreenadh studied the existence and multiplicity of non-negative solutions of $p$-fractional equations with subcritical concave-convex nonlinearities. In \cite{CS}, Chen and Squassina have studied the concave-critical system of equations with the $p-$fractional Laplace operator. More precisely, they studied:
\begin{align*}
 \begin{cases}
  (-\Delta)^s_pu = \la |u|^{q-1}u + \f{2\al}{\al+\ba}|u|^{\al-2}u|v|^{\ba} &\quad\mbox{in}\quad \Omega,\\
  (-\Delta)^s_pv = \la |v|^{q-1}u + \f{2\ba}{\al+\ba}|v|^{\ba-2}v|u|^{\al} &\quad\mbox{in}\quad \Omega,\\
u=v=0 &\quad\mbox{in}\quad\mathbb{R}^N\setminus\Omega,
\end{cases}
\end{align*}
where $\al+\ba=p^*_s$, $0<q<p-1$, $\al,\ \ba>1$, $\la,\  \mu$ are two positive parameters. When $\f{N(p-2)+ps}{N-ps}\leq q<p-1$ and $N>p^2s$, they have proved that there exists $\la_*>0$ such that for $0<\la^\f{p}{p-q}+\mu^\f{p}{p-q}<\la_*$, the above system of equations admits at least two nontrivial solutions.  

Note that, if we set $\la=\mu$, $\al=\ba=\f{p^*_s}{2}$ and $u=v$  then the above system reduces to $(\mathcal{P}_{\mu})$.  Therefore, it follows that when $\f{N(p-2)+ps}{N-ps}\leq q<p-1$ and $N>p^2s$, problem $(\mathcal{P}_{\mu})$ admits two nontrivial solution for $\mu\in(0,\mu_*)$, for some $\mu_*>0$. It can be shown that the nontrivial solutions obtained in \cite{CS} are actually positive solutions of $(\mathcal{P}_{\mu})$ (see Remark \ref{pos} in Section 2).

 \vspace{2mm}
  The main result of this article is the following:
\begin{theorem} \label{thm.2}
Let $\Om$ be a bounded domain with smooth boundary in $\Rn. $ Let $ s \in (0,1)$, $p \geq 2$.   Then
there exist $\mu^* >0$, $N_0\in\N$ and $q_0\in(0,\ p-1)$ such that for all $\mu \in (0,\mu^*)$, $N>N_0$ and $q\in (q_0,\ p-1)$, problem $(\mathcal{P}_\mu)$ 
has at least one sign changing solution, where $N_0$ is given by the following relation:
\begin{align*}
N_{0}:= 
 \begin{cases}
 sp(p+1)\quad\mbox{when} \quad 2\leq p<\f{3+\sqrt{5}}{2},\\
sp(p^2-p+1) \quad\mbox{when}\quad p\geq \f{3+\sqrt{5}}{2}.
   \end{cases}
\end{align*}

\end{theorem}

\vspace{2mm}

{\bf Notations:} Throughout this paper $C$ denotes the generic constant which may vary from line to line. For a Banach space $X$, we denote by $X'$, the dual space of $X$.

\section{\bf Existence of sign-changing solution}

Define the Nehari-manifold $N_\mu$ by
 $$
 N_\mu:=\left\{u\in X_0\setminus \{0\}\,\Big| \langle I_\mu'(u),u\rangle _{X_0}=0\right\}.
 $$
The Nehari manifold $N_\mu$ is closely linked to the behavior of the fibering  map $\varphi_u:(0,\infty)\to\mathbb{R}$ defined by
$$
\varphi_u(r):=I_\mu(ru)=\f{r^p}{p}||u||_{X_0}^p-\f{\mu r^{q+1}}{q+1}|u|_{L^{q+1}(\Om)}^{q+1}
-\f{r^{p^*_s}}{p^*_s}|u|_{L^{p^*_s}(\Om)}^{p^*_s},
$$
which was first introduced by Drabek and Pohozaev in \cite{DP}.
\begin{lemma}\label{N.mu} 
 For any $u\in X_0\setminus \{0\}$, we have $ru\in N_\mu$ if and only if $\varphi_u'(r)=0.$ 
\end{lemma}
\begin{proof}
 We note that for $r>0$, $\varphi_u'(r)=\langle I_\mu'(ru),u\rangle _{X_0}=\f{1}{r}\langle I_\mu'(ru),ru\rangle _{X_0}$.
 Hence, $\varphi_u'(r)=0$  if and only if $ru\in N_\mu.$
 \end{proof}
Therefore, we can conclude that the elements in $N_\mu$ corresponds to the stationary point of the maps $\varphi_u.$ Observe that
\begin{align}\label{phi'}
\varphi_u'(r)=r^{p-1}||u||_{X_0}^p-\mu r^q|u|_{L^{q+1}(\Om)}^{q+1}-r^{p^*_s-1}|u|_{L^{p^*_s}(\Om)}^{p^*_s}
\end{align}
and
\begin{align}\label{phi''}
\varphi_u''(r)=(p-1)r^{p-2}||u||_{X_0}^p-q\mu r^{q-1}|u|_{L^{q+1}(\Om)}^{q+1}
-(p^*_s-1)r^{p^*_s-2}|u|_{L^{p^*_s}(\Om)}^{p^*_s}.
\end{align}
By Lemma \ref{N.mu}, we note that $u\in N_\mu$ if and only if $\varphi_u'(1)=0.$ Hence for $u\in N_\mu$, using (\ref{phi'}) and 
(\ref{phi''}), we obtain that 
\begin{align}\lab{Mar-23-1}
\varphi_u''(1)&=(p-1)||u||_{X_0}^p-q\mu |u|_{L^{q+1}(\Om)}^{q+1}-(p^*_s-1)|u|_{L^{p^*_s}(\Om)}^{p^*_s}\no\\
              &=(p-p^*_s)|u|_{L^{p^*_s}(\Om)}^{p^*_s}+(1-q)\mu |u|_{L^{q+1}(\Om)}^{q+1}\no\\
              &=(p-1-q)||u||_{X_0}^p-(p^*_s-1-q)|u|_{L^{p^*_s}(\Om)}^{p^*_s}\\
              &=(p-p^*_s)||u||_{X_0}^p+(p^*_s-1-q)\mu |u|_{L^{q+1}(\Om)}^{q+1}\no.
\end{align}
Therefore, we split the manifold into three parts corresponding to local minima, maxima and points of inflection
\begin{align*}
 N_\mu^+&:=\left\{u\in  N_\mu\,\Big| \varphi_u''(1)>0\right\},\\
  N_\mu^-&:=\left\{u\in  N_\mu\,\Big| \varphi_u''(1)<0\right\},\\
   N_\mu^0&:=\left\{u\in  N_\mu\,\Big| \varphi_u''(1)=0\right\}.
\end{align*}

\vspace{3mm}

\begin{remark}\lab{pos}From \cite{CS}, it follows that $\inf_{u \in N_\mu^+} I_{\mu}(u)$ and
$\inf_{u \in N^-_\mu} I_{\mu}(u)$ are achieved and those two infimum points are two critical points of $I_{\mu}$.
 Now if we define $I^+_{\mu}$ as follows:
\be\lab{Feb-18-1} I^+_{\mu}(u):= \frac{1}{p}\norm{u}_{X_0}^p
 -\frac{\mu}{q+1}|u^+|_{L^{q+1}(\Om)}^{q+1}-\frac{1}{p^*_s}|u^+|_{L^{p^*_s}(\Om)}^{p^*_s}\ee and
 \be\lab{alpha-mu}
\tilde\al_{\mu}^+:=\inf_{u \in N_\mu^+}I^+_{\mu}(u) \quad\text{and}
\quad \tilde\al_{\mu}^-:=\inf_{u \in N^-_\mu} I^+_{\mu}(u),\ee  then repeating the same analysis as in \cite{CS} for $I^+_{\mu}$, it can be shown that there exists $\mu_*>0$ such that for $\mu\in(0,\mu_*)$, there exists two non-trivial critical points $w_0\in N_{\mu}^+$ and  $w_1\in N_{\mu}^-$ of $I^+_{\mu}$. It is not difficult to see that $w_0$ and $w_1$ are nonnegative in $\Rn$. Indeed,
\begin{align}
&0=\<(I_\mu^+)' (w_0),w_0^{-}\>\no\\
&=\int_{\R^{2N}}\f{|w_0(x)-w_0(y)|^{p-2}(w_0(x)-w_0(y))(w_0^-(x)-w_0^-(y))}{|x-y|^{N+sp}}dxdy\no\\
&=\int_{\R^{2N}}\f{|w_0(x)-w_0(y)|^{p-2}((w_0^-(x)-w_0^-(y))^2+2(w_0^-(x)w_0^+(y)))}{|x-y|^{N+sp}}dxdy\no\\
&\geq \int_{\R^{2N}}\f{|w_0^-(x)-w_0^-(y)|^{p}}{|x-y|^{N+sp}}dxdy=\norm{w_0^-}^p_{X_0}.\no\\
\end{align}
Thus, $\norm{w_0^-}_{X_0}=0$ and hence, $w_0=w_0^+.$ Similarly we can show $w_1=w_1^+.$ Using maximum principle 
 \cite[Theorem A.1]{BraFra}
we conclude that  both $w_0, w_1$ are positive almost everywhere in $\Om$. Hence $(\mathcal{P}_{\mu})$ has at least two positive solutions. \end{remark}

\vspace{3mm}

Set 
\be\lab{mu'}
\tilde\mu=\bigg(\frac{p-1-q}{p^*_s-q-1}\bigg)^{\frac{p-1-q}{p^*_s-p}}\frac{p^*_s-p}{p^*_s-q-1}
|\Om|^{\frac{q+1-p^*_s}{p^*_p}}S^{\frac{N(p-1-q)}{p^2s}+\frac{q+1}{p}}. 
\ee
Next we prove three elementary lemmas. 
\begin{lemma}\label{N.mu-i}
Let $\mu \in (0,\tilde\mu). $ For every $u \in X_0,\  u \neq 0, $ there exists unique 
$$t^-(u)<t_0(u)=
\bigg(\frac{(p-1-q)||{u}||_{X_0}^p}{(p^*_s-1-q)|u|^{p^*_s}_{L^{p^*_s}(\Om)}}\bigg)^{\frac{N-ps}{p^2s}}<t^+(u), $$ such that 
\begin{align}
&t^-(u)u \in N^+_\mu \quad\mbox{and}\quad I_{\mu}(t^-u)=\min_{t \in [0,t_0]}I_\mu(tu), \notag\\ 
&t^+(u)u \in N^-_\mu \quad\mbox{and}\quad I_{\mu}(t^+u)=\max_{t \geq t_0}I_\mu(tu). \notag
\end{align}
\end{lemma}
\begin{proof}
 
For $t \geq 0$, 
 \be
  I_{\mu}(tu)=\frac{t^p}{p}||{u}||_{X_0}^p-\frac{\mu t^{q+1}}{q+1}|u|^{q+1}_{L^{q+1}(\Om)}
  -\frac{t^{p^*_s}}{p^*_s}|u|^{p^*_s}_{L^{p^*_s}(\Om)}. \no
  \ee
 Therefore \be\frac{\pa}{\pa t}I_{\mu}(tu)=t^q\bigg(t^{p-1-q}||{u}||_{X_0}^p-t^{p^*_s-q-1}|u|^{p^*_s}_{L^{p^*_s}(\Om)}
 -\mu|u|^{q+1}_{L^{q+1}(\Om)}\bigg) .\no\ee
Define 
\be\lab{eq:psi}
\psi(t)=t^{p-1-q}||{u}||_{X_0}^p-t^{p^*_s-q-1}|u|^{p^*_s}_{L^{p^*_s}(\Om)}. 
\ee
By a straight forward computation, it follows that $\psi$ attains maximum at  
the point 
\be\lab{r-0}
t_0=t_0(u)=\bigg(\frac{(p-1-q)||{u}||_{X_0}^p}{(p^*_s-1-q)|u|^{p^*_s}_{L^{p^*_s}(\Om)}}\bigg)^{\frac{1}{p^*_s-p}} .
\ee
 Thus
\be\lab{2-9} \psi'(t_0)=0, \quad \psi'(t) >0 \quad\mbox{if}\quad t<t_0, \quad \psi'(t) <0 \quad\mbox{if}\quad t>t_0.\ee
Moreover, $\psi(t_0)=\left(\frac{p-1-q}{p^*_s-1-q}\right)^{\frac{p-1-q}{p^*_s-p}}\left(\frac{p^*_s-p}{p^*_s-1-q}\right)
\left(\frac{||{u}||_{X_0}^{p(p^*_s-1-q)}}{|u|^{p^*_s(p-1-q)}_{L^{p^*_s}(\Om)}}\right)^{\frac{N-ps}{p^2s}}. $
Therefore using Sobolev embedding, we have
\be\lab{2-10}
\psi(t_0) \geq \displaystyle\left(\frac{p-1-q}{p^*_s-1-q}\right)^{\frac{(p-1-q)(N-2s)}{4s}}
\left(\frac{p^*_s-p}{p^*_s-1-q}\right)S^{\frac{N(p-1-q)}{p^2s}}||{u}||_{X_0}^{q+1}. 
\ee
Using H\"older inequality followed by Sobolev inequality,  and the fact that $\mu\in (0,\tilde\mu)$, we obtain 
\be
\mu \Iom |u|^{q+1}dx \leq \mu ||{u}||_{X_0}^{q+1}S^{-(q+1)/p}|\Om|^{\frac{p^*_s-q-1}{p^*_s}}
\leq \tilde\mu ||{u}||_{X_0}^{q+1}S^{-(q+1)/p}|\Om|^{\frac{p^*_s-q-1}{p^*_s}} \leq \psi(t_0),\no
\ee
where in the last inequality we have used expression of $\tilde\mu$ (see \eqref{mu'}) and \eqref{2-10}.
Hence, there exists $t^+(u)>t_0>t^-(u)$ such that
\be\lab{2-11}\psi(t^+)=\mu\Iom|u|^{q+1}=\psi(t^-)\quad\mbox{and}\quad \psi'(t^+)<0<\psi'(t^-). \ee
This in turn, implies $t^+u \in N^-_\mu$ and  $t^-u \in N^+_\mu $. Moreover,
using \eqref{2-9} and \eqref{2-11} in the expression of $\frac{\pa}{\pa t}I_{\mu}(tu)$, we have
$$\quad \frac{\pa}{\pa t}I_{\mu}(tu) > 0 \quad\mbox{when}\quad t\in(t^-,t^+) \quad\text{and}
\quad \frac{\pa}{\pa t}I_{\mu}(tu) < 0 \quad\mbox{when}\quad t\in[0,t^-)\cup(t^+,\infty),
$$
$$\frac{\pa}{\pa t}I_{\mu}(tu)=0 \quad\text{when}\quad t=t^{\pm}.$$
We note that $I_{\mu}(tu)=0$ at $t=0$ and strictly negative when $t>0$ is small enough. Therefore it is easy to conclude that 
$$\max_{t \geq t_0}I_\mu(tu)=I_\mu(t^+u) \quad\mbox{and} \quad \min_{t \in [0,t_0]}J_\mu(tu)=I_\mu(t^-u).$$
\end{proof}

Repeating the same argument as in Lemma \ref{N.mu-i}, we can also prove that the following lemma holds:

\begin{lemma}\label{N.mu-i-2}
Let  $\mu \in (0,\tilde\mu)$, where $\tilde\mu$ is defined as in $\eqref{mu'}$.  For every $u \in X_0,\  u \neq 0, $ there exist unique 
$$\tilde{t}^-(u)<\tilde{t}_0(u)=
\bigg(\frac{(p-1-q)||{u}||_{X_0}^p}{(p^*_s-1-q)|u^+|^{p^*_s}_{L^{p^*_s}(\Om)}}\bigg)^{\frac{N-ps}{p^2s}}<\tilde{t}^+(u), $$ such that 
\begin{align}
&\tilde{t}^-(u)u \in N^+_\mu \quad\mbox{and}\quad I^+_{\mu}(\tilde{t}^-u)=\min_{t \in [0,t_0]}I^+_\mu(tu), \notag\\ 
&\tilde{t}^+(u)u \in N^-_\mu \quad\mbox{and}\quad I^+_{\mu}(\tilde{t}^+u)=\max_{t \geq t_0}I^+_\mu(tu), \notag
\end{align}
where $I_{\mu}^+$ is defined as in \eqref{Feb-18-1}.
\end{lemma}

\begin{lemma}\label{N.mu-ii}
Let $\tilde\mu$ be defined as in $\eqref{mu'}$.  Then $\mu \in (0,\tilde\mu), $ implies $N^0_\mu=\emptyset$. 
\end{lemma}
\begin{proof}
 Suppose not. Then there exists $ w \in N_\mu^0$ such that $w\not=0$ and
 \begin{align}\label{2-12}
(p-1-q)||{w}||_{X_0}^p-(p^*_s-q-1)|w^+|^{p^*_s}_{L^{p^*_s}(\Om)}=0.
\end{align}
The above expression combined with Sobolev inequality yields
\begin{align}\label{2-13}
||{w}||_{X_0} \geq S^{\frac{N}{p^2s}}\displaystyle\left(\frac{p-1-q}{p^*_s-1-q}\right)^{\frac{N-ps}{p^2s}}.
\end{align}
As $w \in N_{\mu}^0 \subseteq N_{\mu}$, using \eqref{2-12} and H\"older inequality followed by Sobolev inequality, we get
\Bea
 0&=&||{w}||_{X_0}^p-|w|^{p^*_s}_{L^{p^*_s}(\Om)}-\mu|w|^{q+1}_{L^{q+1}(\Om)}\\
 &\geq&||{w}||_{X_0}^p-\displaystyle\left(\f{p-1-q}{p^*_s-q-1}\right)||w||_{X_0}^p
 -\mu|\Om|^{1-\frac{q+1}{p^*_s}}S^{-(q+1)/p}||w||_{X_0}^{q+1}.
\Eea 
Combining the above inequality with \eqref{2-13} and using $\mu<\tilde\mu$, we have
$$0\geq||{w}||_{X_0}^{q+1}\displaystyle\left[\bigg(\frac{p^*_s-p}{p^*_s-q-1}\bigg)
\bigg(\frac{p-1-q}{p^*_s-q-1}\bigg)^{\frac{(N-ps)(p-1-q)}{p^2s}}S^{\frac{N(p-1-q)}{p^2s}}-
 \mu|\Om|^{1-\frac{q+1}{p^*_s}}S^{-(q+1)/p}\right]
 > 0,$$
which is a contradiction. This completes the proof.
\end{proof}

\begin{lemma}\label{N.mu-iii}
Let $\tilde\mu$ is as defined in $\eqref{mu'}$ and $\mu \in (0,\tilde\mu)$. 
Given $u \in N_\mu^{-}, $ there exists $\rho_u>0$ and a differentiable function $g_{\rho_u}:B_{\rho_u}(0) \to \R^+$
satisfying the following:
\begin{align}
 &g_{\rho_u}(0)=1, \notag\\
 &\big(g_{\rho_u}(w)\big)(u+w) \in N_\mu^{-}\quad \forall\quad w \in B_{\rho_u}(0), \notag\\
 &\<g'_{\rho_u}(0),\phi\>=\frac{p A(u,\phi)-p^*_s\displaystyle\Iom |u|^{p^*_s-2}u\phi-(q+1)\mu\displaystyle\Iom |u|^{q-1}u\phi}{(p-1-q)||{u}||_{X_0}^p-
       (p^*_s-q-1)|u|^{p^*_s}_{L^{p^*_s}(\Om)}}\quad\forall\,\phi\in B_{\rho_u}(0),  \notag
\end{align}
where 
$$
A(u,\phi)=\int_{\R^{2N}}\frac{|u(x)-u(y)|^{p-2}(u(x)-u(y))(\phi(x)-\phi(y))}{|x-y|^{N+ps}}dxdy.
$$
\end{lemma}
\begin{proof}
 Define $E:\R \times X_0 \to \R$ as follows:
 $$E(r,w)=r^{p-1-q}||{u+w}||_{X_0}^p-r^{p^*_p-q-1}|(u+w)|^{p^*_s}_{L^{p^*_s}(\Om)}-\mu |(u+w)|^{q+1}_{L^{q+1}(\Om)}. $$
We note that $u \in N_\mu^{-}\subset N_{\mu}$ implies
 $$E(1,0)=0,  \quad\text{and}\quad \frac{\pa E}{\pa r}(1,0)=(p-1-q)||{u}||_{X_0}^p
 -(p^*_s-q-1)|u|^{p^*_s}_{L^{p^*_s}(\Om)}<0. $$
Therefore, by implicit function theorem, there exists neighborhood $B_{\rho_u}(0)\subset N_\mu$ 
for some $\rho_u>0$ and a $C^1$ function
 $g_{\rho_u}:B_{\rho_u}(0) \to \R^+$ such that
 \begin{align}
  &(i)\ g_{\rho_u}(0)=1, \quad (ii)\ E(g_{\rho_u}(w),w)=0, \,\, \forall\  w \in B_{\rho_u}(0), \no\\ 
  &(iii) E_r(g_{\rho_u}(w),w)<0, \,\, \forall\  w \in B_{\rho_u}(0), \quad
  (iv)\  \<g'_{\rho_u}(0),\phi\>=-\frac{\<\frac{\pa E}{\pa w}(1,0),\phi\>}{\frac{\pa E}{\pa r}(1,0)}. \notag
 \end{align}
 Multiplying (ii) by $(g_{\rho_u}(w))^{q+1}$, it follows that $g_{\rho_u}(w)(u+w)\in N_\mu$.
 In fact, simplifying (iii), we obtain
 $$(p-1-q)g_{\rho_u}(w)^p||u+w||_{X_0}^p-(p^*_s-q-1)g_{\rho_u}(w)^{p^*_s}|(u+w)|_{p^*_s}^{p^*_s}<0 \quad\forall\  w \in B_{\rho_u}(0).$$ Thus $\big(g_{\rho_u}(w)\big)(u+w)\in  N^-_\mu$, 
 for every $  w \in B_{\rho_u}(0)$. The last assertion of the lemma follows from (iv).
\end{proof}

\vspace{2mm}
 
Let $S$ be as in \eqref{S}. From \cite{BMS}, we know that for $1 < p < \infty,s\in (0, 1), N > ps,$ there exists a minimizer
for $S,$ and for every minimizer $U,$ there exist $x_0\in\Rn$ and a constant sign monotone function $u:\mathbb{R}\to
\mathbb{R}$ such that $U(x)=u(|x-x_0 |).$ In the following, we shall fix a radially symmetric nonnegative decreasing
minimizer $U=U(r)$ for $S.$ Multiplying $U$ by a positive constant if necessary, we may assume that
\begin{align}\lab{min}
 (-\Delta)_p^sU=U^{p^*_s-1}\quad\mbox{in}\quad\mathbb{R}^n
\end{align}
For any $\eps>0$ we note that the function function 
\be\lab{U-eps}U_\eps(x)=\frac{1}{\eps^{\frac{(N-sp)}{p}}}U\left(\frac{|x|}{\eps}\right)\ee is also a minimizer for $S$ satisfying \eqref{min}.
From \cite{MPSY}, we also have the following asymptotic estimates for U.\begin{lemma}\label{decay}\cite{MPSY}
Let $U$ be the solution of \eqref{min}. Then, there exists $c_1,c_2>0$ and $\theta>1$ such that for all $r \geq 1,$ 
 \begin{align}\label{+}
  \frac{c_1}{r^{\frac{N-sp}{p-1}}} \leq U(r) \leq \frac{c_2}{r^{\frac{N-sp}{p-1}}}
 \end{align}
 and
 \begin{align}\label{++}
  \frac{U(r \theta)}{U(r)} \leq \frac{1}{2}.
 \end{align}
\end{lemma}
\begin{proof}
 See [lemma 2.2 \cite{MPSY}].
\end{proof}
Therefore we have,
\be\lab{Jan-23-1}
c_1\f{\eps^\f{N-sp}{p(p-1)}}{|x|^\f{N-sp}{p-1}}\leq U_{\eps}(x)\leq c_2 \f{\eps^\f{N-sp}{p(p-1)}}{|x|^\f{N-sp}{p-1}} \quad\text{ for}\quad |x|>\eps.
\ee

We consider a cut-off function $\psi\in C_0^{\infty}(\Om)$ such that 
 $0\leq\psi \leq 1$, $\psi \equiv 1$ in $\Om_\de$, $\psi \equiv 0$ in $\Rn\setminus \Om$, where 
 $$\Om_{\de}:=\{x\in\Om: \text{dist}(x,\pa\Om)>\de\}.$$
Define 
\be\lab{u-eps}
u_{\eps}(x)=\psi(x)U_{\eps}(x).
\ee

We need the following lemmas in order to prove Theorem \ref{thm.2}. 
\begin{lemma}\lab{l:4i}
Suppose $w_1$ is a positive solution of $(P_\mu)$ and $u_{\eps}$ is as defined in \eqref{u-eps}. Then for every $\eps>0$, small enough

\begin{itemize}
\item[(i)] $A_1 :=\displaystyle\Iom w_1^{p^*_s-1}u_{\eps}dx\leq k_1\eps^\f{N-ps}{p(p-1)}$;
\item[(ii)] $A_2 :=\displaystyle\Iom w_1^{q}u_{\eps}dx\leq k_2\eps^\f{N-ps}{p(p-1)}$;
\item[(iii)] $A_3 :=\displaystyle\Iom w_1u_{\eps}^{q}dx\leq k_3\eps^{\f{N-ps}{p(p-1)}q}$;
\item[(iv)]$A_4 :=\displaystyle\Iom w_1u_{\eps}^{p_s^*-1}dx\leq k_4\eps^\f{N(p-1)+ps}{p(p-1)}$.
\end{itemize} 
\end{lemma}
\begin{proof} Applying the Moser iteration technique (see \cite[Theorem 3.3]{BrLinPa}), it can be shown that any positive solution of $(P_{\mu})$ is in $L^{\infty}(\Om)$ . Let $R,\  M>0$ be such that $\Om\subset B(0, R)$ and $|w_1|_{L^{\infty}(\Om)}<M$. 
\Bea
(i) \quad A_1=\displaystyle\Iom w_1^{p^*_s-1}u_{\eps}dx &\leq& 
C\bigg[\int_{\Om\cap\{|x|\leq\eps\}}U_{\eps}(x)dx+\eps^\f{N-sp}{p(p-1)}\int_{\Om\cap\{ |x|>\eps\}}\f{dx}{|x|^\f{N-sp}{p-1}}\bigg]\\
&\leq& C\bigg[\eps^{N-\f{(N-sp)}{p}}\int_{\{|x|<1\}}U(x)dx   +\eps^\f{N-sp}{p(p-1)}\int_{ B(0,R)}\f{dx}{|x|^\f{N-sp}{p-1}}dx\bigg]\\
&\leq& C\bigg[\eps^{N-\f{(N-sp)}{p}}+\eps^\f{N-sp}{p(p-1)}\int_0^R r^{N-1-\f{N-sp}{p-1}}dr\bigg]\\
&\leq& k_1\eps^\f{N-sp}{p(p-1)}.
\Eea
Proof of (ii) similar to (i).\\
\Bea
(iii) \quad A_3=\displaystyle\Iom w_1u_{\eps}^qdx &\leq&
C\bigg[\int_{\Om\cap\{|x|\leq\eps\}}U^q_{\eps}(x)dx+\eps^{\f{N-sp}{p(p-1)}q}\int_{\Om\cap\{ |x|>\eps\}}\f{dx}{|x|^\f{(N-sp)q}{p-1}}\bigg]\\
&\leq& C\bigg[\eps^{N-\f{(N-sp)q}{p}}\int_{\{|x|<1\}}U(x)^qdx   +\eps^\f{(N-sp)q}{p(p-1)}\int_{ B(0,R)}\f{dx}{|x|^\f{(N-sp)q}{p-1}}dx\bigg]\\
&\leq& C\bigg[\eps^{N-\f{(N-sp)q}{p}}+\eps^{\f{N-sp}{p(p-1)}q}\int_0^R r^{N-1-{\f{N-sp}{p-1}q}}dr\bigg]\\
&\leq& k_3\eps^{\f{N-ps}{p(p-1)}q},
\Eea
since $0<q<p-1<\f{N(p-1)}{N-sp}$.
(iv) can be proved as in (iii).
\end{proof}
\begin{lemma}\lab{l:4ii}
Let $u_{\eps}$ be as defined in \eqref{u-eps},  $0<q<p-1$ and $N > p^2s$. Then for every $\eps>0$, small \\

$\displaystyle\Iom |u_{\epsilon}|^{q+1}dx\geq\left\{\begin{array}{lll}
k_5\eps^\f{(N-ps)(q+1)}{p(p-1)} \quad &\text{if}\quad  0<q<\f{N(p-2)+ps}{N-ps},\\
k_6\eps^\f{N}{p}|\mbox{ln}\ \eps|,  \quad &\text{if}\quad  q=\f{N(p-2)+ps}{N-ps},\\ 
k_7\eps^{N-\f{(N-ps)(q+1)}{p}}  \quad &\text{if}\quad  \f{N(p-2)+ps}{N-ps}<q<p-1.
\end{array}
\right.$
\end{lemma}
\begin{proof} 
We recall that $R'>0$ was chosen such that $B(0, R')\subset \Om_{\de}$. Therefore, for $\eps>0$ small,  we have


\bea\lab{Jan-23-2}
 \int_{\Omega}|u_{\eps}|^{q+1}dx &\geq& \int_{B(0, R')}|u_{\eps}|^{q+1}dx\no\\
  &=& \int_{B(0, R')}U_{\eps}^{q+1}(x)dx\no\\
 &=& C\eps^{N-\f{(N-sp)(q+1)}{p}}\int_{B\left(0,\f{R'}{\eps}\right)} U^{q+1}(y)dy\\
&\geq& C\eps^{N-\f{(N-ps)(q+1)}{p}} \int_{B\left(0,\f{R'}{\eps}\right)\setminus B(0,1)} U^{q+1}(y)dy \no\\
 & \geq& C\eps^{N-\f{(N-ps)(q+1)}{p}}\int_1^{\f{R'}{\eps}}r^{{N-1}-\f{(N-ps)(q+1)}{p-1}}dr.
 \eea

{\it Case 1} : $0<q\leq\f{N(p-2)+ps}{N-ps}$.\\
We note that
\be\lab{4-2}
\int_1^{\f{R'}{\eps}} r^{(N-1)-\f{(N-ps)(q+1)}{p-1}}dr
\geq C_1\eps^{-N+\f{(N-ps)(q+1)}{p-1}}-C_2,
 \ee
Thus substituting back in (2.17), we obtain
\bea 
\int_{\Omega}|u_{\eps}|^{q+1}dx&\geq& C\eps^{N-\f{(N-ps)(q+1)}{p}}[C_1\eps^{-N+\f{(N-ps)(q+1)}{p-1}}-C_2]\no\\
&=&C_3\eps^\f{(N-ps)(q+1)}{p(p-1)}-C_4\eps^{N-\f{(N-ps)(q+1)}{p}}\no\\
&\geq& k_5\eps^\f{(N-ps)(q+1)}{p(p-1)}.
\eea
{\it Case 2} :  $q=\f{N(p-2)+ps}{N-ps}$.\\
In this case it follows
\begin{align*}
\int_1^{\f{R'}{\eps}}r^{{N-1}-\f{(N-ps)(q+1)}{p-1}}dr &\geq C|\mbox{ln}\ \eps|.
\end{align*} 

Plugging back in (2.17), we obtain
\be 
\int_{\Omega}|u_{\eps}|^{q+1}dx\geq k_6\eps^{N-\f{(N-ps)(q+1)}{p}}|\ln \eps|=k_6\eps^\f{N}{p}|\ln \eps|.\no
\ee

{\it Case 3} :  $\f{N(p-2)+ps}{N-ps}<q<p-1$.\\
\bea\lab{4-3}
\text{RHS of (2.16)} &\geq& k_7\eps^{N-\f{(N-sp)(q+1)}{p}}\int_{B(0, 1)}U^{q+1}(x)dx\no\\
 &\geq&k_7\eps^{N-\f{(N-sp)(q+1)}{p}}. 
 \eea
Hence the lemma follows.

\end{proof}

\begin{definition}
We say $\{u_n\}$  is a Palais Smale (PS) sequence of $I_{\mu}$ at level $c$ (in short $(PS)_c$) if  $I_{\mu}(u_n)\to c$ and $I_{\mu}'(u_n)\to 0$ in $(X_0)'$. Furthermore, we say $I_{\mu}$ satisfies Palais-Smale condition at level $c$ if for all $\{u_n\}\subset X_0$ with $I_{\mu}(u_n)\to c$ and $I_{\mu}'(u_n)\to 0$ in $(X_0)'$, implies up to a subsequence $u_n$ converges strongly in $X_0$.
\end{definition}

Let us define
\begin{align}\label{M}
M:=\f{(pN-(N-ps)(q+1))(p-1-q)}{p^2(q+1)}\left(\f{(p-1-q)(N-sp)}{p^2s}\right)^{\f{q+1}{p^*_s-q-1}}|\Om|.
\end{align} 

\begin{lemma}\label{ps}
Let $M$ be as in \eqref{M}.
For any $\mu>0,$ and for 
 $$
 c<\f{s}{N}S^{\f{N}{sp}}-M\mu^{\f{p^*_s}{p^*_s-q-1}},
 $$
 $I_\mu$ satisfies $(PS)_c$ condition.
\end{lemma}
\begin{proof}
Let $\{u_k\}\subset X_0$ be a $(PS)_c$ sequence for $I_\mu$,  that is, we have $I_\mu(u_k)\to c$ and $I_\mu'(u_k)\to0$ in $(X_0)'$
as $k\to\infty$. By the standard method it is not difficult to see that  $\{u_k\}$ is bounded in $X_0$. Then up to a subsequence, still denoted by $u_k,$ there exists
 $u_\infty \in X_0$ such that 
 \begin{align*}
 & u_k \rightharpoonup u_\infty\quad\mbox{weakly in}\quad X_0\quad\mbox{as}\quad k \to \infty,\\
 & u_k \rightharpoonup  u_\infty\quad\mbox{weakly in}\quad L^{p^*_s}(\Rn)\quad\mbox{as}\quad  k \to \infty,\\
 & u_k \to u_\infty \quad\mbox{strongly in}\quad L^r(\Rn)\quad\mbox{for any}\quad 1 \leq r <p^*_s\quad\mbox{as}\quad  k \to \infty,\\
 & u_k \to u_\infty \quad\mbox{a.e. in}\quad \Rn\quad\mbox{as}\quad  k \to \infty. 
 \end{align*}
As $0<q< p-1$,   we have 
 \begin{align*}
  \Iom |u_k|^{q+1}(x)dx \to \Iom |u_\infty|^{q+1}(x)dx \quad \mbox{as}\quad k \to \infty. 
 \end{align*}
Using these above properties it can be shown that $\<I'_\mu(u_\infty),\varphi\>_{X_0}=0$ for any $\varphi\in X_0.$

\noi Indeed for any $\varphi\in X_0$, 
\begin{eqnarray*}
 \langle I_\mu'(u_k),\varphi\rangle - \langle I_\mu'(u_\infty),\varphi\rangle
 &=&\int_{\R^{2N}}\f{|u_k(x)-u_k(y)|^{p-2}(u_k(x)-u_k(y))(\varphi(x)-\varphi(y))}{|x-y|^{N+sp}}dxdy \\
 & -&\int_{\R^{2N}}\f{|u_\infty(x)-u_\infty(y)|^{p-2}(u_\infty(x)-u_\infty(y))(\varphi(x)-\varphi(y))}{|x-y|^{N+sp}}dxdy \\
 &-&\mu\left(\int_\Om|u_k|^{q-1}u_k\varphi\ dx-\int_\Om|u_\infty|^{q-1}u_\infty\varphi\ dx\right) \\
 &- &\left(\int_\Om|u_k|^{p^*_s-2}u_k\varphi\ dx-\int_\Om|u_\infty|^{p^*_s-2}u_\infty\varphi\ dx\right).
\end{eqnarray*}
As $\left\{\f{|u_k(x)-u_k(y)|^{p-2}(u_k(x)-u_k(y))}{|x-y|^{\f{N+sp}{p'}}}\right\}_{k\geq 1}$ is bounded in $L^{p'}(\R^{2N})$, where $p'=\f{p}{p-1}$, 
upto a subsequence 
\begin{align*}
&\f{|u_k(x)-u_k(y)|^{p-2}(u_k(x)-u_k(y))}{|x-y|^{\f{N+sp}{p'}}} \rightharpoonup
\f{|u_{\infty}(x)-u_{\infty}(y)|^{p-2}(u_{\infty}(x)-u_{\infty}(y))}{|x-y|^{\f{N+sp}{p'}}}
\end{align*}
weakly in  $L^{p'}(\R^{2N})$ ,
$u_k \rightharpoonup  u_\infty$ weakly in $L^{p^*_s}(\Rn)$  and
$u_k \to u_\infty$ strongly in  $L^{q+1}(\Rn)$ as $k \to \infty$.

Combining these we have $\langle I_\mu'(u_k),\varphi\rangle - \langle I_\mu'(u_\infty),\varphi\rangle\to 0 $ as
$k\to\infty.$ But as $I'_\mu(u_k)\to 0$ in $X_0'$ as $k\to\infty,$ we have $\<I'_\mu(u_\infty),\varphi\>_{X_0}=0$ for any $\varphi\in X_0.$
Hence, in particular $\<I'_\mu(u_\infty),u_\infty\>_{X_0}=0$.

Furthermore, by Brezis-Lieb lemma as $k \to \infty, $ we get,
\begin{align*}
 \int_{\R^{2N}}\frac{|u_k(x)-u_k(y)|^p}{|x-y|^{N+sp}}dxdy=\int_{\R^{2N}}\frac{|u_k(x)-u_\infty(x)-u_k(y)+u_\infty(y)|^p}{|x-y|^{N+sp}}dxdy\\
 +\int_{\R^{2N}}\frac{|u_\infty(x)-u_\infty(y)|^p}{|x-y|^{N+sp}}dxdy+o(1)
\end{align*}
and
\begin{align*}
 \Iom |u_k(x)|^{p^*_s}dx=\Iom|(u_k-u_\infty)(x)|^{p^*_s}dx+\Iom |u_\infty(x)|^{p^*_s}dx+o(1).
\end{align*}
Now, 
\Bea
\<I'_\mu(u_k),u_k\>_{X_o}&=& \int_{\R^{2n}}\frac{|u_k(x)-u_k(y)|^p}{|x-y|^{N+sp}}dxdy-\mu\Iom |u_k(x)|^{q+1}dx
-\Iom |u_k(x)|^{p^*_s}dx\\
&=& \int_{\R^{2n}}\frac{|u_k(x)-u_\infty(x)-u_k(y)+u_\infty(y)|^p}{|x-y|^{N+sp}}dxdy-\Iom |u_k(x)-u_\infty(x)|^{p^*_s}dx\\
\qquad &+& \<I'_\mu(u_\infty),u_\infty\>_{X_0}+o(1).\\
\Eea
Since as $\<I'_\mu(u_\infty),u_\infty\>_{X_0}=0$ and $\<I'_\mu(u_k),u_k\>_{X_0} \to 0$ as $k \to \infty, $
we have that there exists $b\in \R$ with $b\geq 0$ such that
\begin{align}
 ||{u_k-u_\infty}||_{X_0}^p=\int_Q\frac{|u_k(x)-u_\infty(x)-u_k(y)+u_\infty(y)|^p}{|x-y|^{N+sp}}dxdy \to b
\end{align}
and 
\begin{align}
 \Iom|(u_k-u_\infty)(x)|^{p^*_s}dx \to b \quad\mbox{as}\quad k \to \infty. 
\end{align}
If $b=0,$ we are done. Suppose  $b>0$.   Moreover, using Sobolev inequality we have,
$$||{u_k-u_\infty}||_{X_0}^p \geq S\left(\Iom(|u_k-u_\infty)(x)|^{p^*_s}dx \right)^{p/{p^*_s}}.$$
Therefore, $b \geq Sb^{p/{p^*_s}}$, and this implies $b \geq S^{N/sp}. $
On the other hand, \\
since  $\<I'_\mu(u_\infty),u_\infty\>_{X_0}=0$ 
we obtain 
\begin{align}\label{eq1}
I_\mu(u_\infty)&=I_\mu(u_\infty)-\f{1}{p} \<I'_\mu(u_\infty),u_\infty\>_{X_0}\no\\
&=\frac{s}{N}\Iom |u_\infty(x)|^{p^*_s}dx+\mu\left(\frac{1}{p}-\frac{1}{q+1}\right)\Iom |u_\infty(x)|^{q+1}dx.
\end{align}
Using \eqref{eq1} and $\<I'_\mu(u_k),u_k\>_{X_0} \to 0$ as $k \to \infty, $ we get
\begin{align}\label{eq2}
c&= \lim_{k \to \infty}I_\mu(u_k)=\lim_{k \to \infty}[I_\mu(u_k)-\frac{1}{p}\<I'_\mu(u_k),u_k\>_{X_0}]\no\\
&= \lim_{k \to \infty}\left[\frac{s}{N}\Iom |(u_k-u_\infty)(x)|^{p^*_s}dx
+\frac{s}{N}\Iom |u_\infty(x)|^{p^*_s}dx +\mu\left(\frac{1}{p}-\frac{1}{q+1}\right)\Iom |u_k(x)|^{q+1}dx\right]\no\\
&=\frac{s}{N}b+\frac{s}{N}\Iom |u_\infty(x)|^{p^*_s}dx+\mu\left(\frac{1}{p}-\frac{1}{q+1}\right)
\Iom |u_\infty(x)|^{q+1}dx\no\\
&\geq\frac{s}{N}S^{N/sp}+\frac{s}{N}\Iom |u_\infty(x)|^{p^*_s}dx+\mu\left(\frac{1}{p}-\frac{1}{q+1}\right)
\Iom |u_\infty(x)|^{q+1}dx\\
&= \frac{s}{N}S^{N/sp}+I_\mu(u_\infty).
\end{align}

Since, by  assumption we have $c<\frac{s}{N}S^{N/sp}$, the last inequality implies $I_\mu(u_\infty)<0. $ In particular, $u_\infty \nequiv 0$ and 
$$0<\frac{1}{p}||{u_\infty}||_{X_0}^p<\frac{\mu}{q+1}\Iom (u_\infty(x))^{q+1}dx+\frac{1}{p^*_s}\Iom (u_\infty(x))^{p^*_s}dx.$$

Moreover, by H\"older inequality we have,
$$\Iom |u_\infty(x)|^{q+1}dx \leq |\Om|^{\frac{p^*_s-(q+1)}{p^*_s}}\left(\Iom |u_\infty(x)|^{p^*_s}dx\right)^{\frac{q+1}{p^*_s}}.$$
Thus, from (\ref{eq2})
\Bea
c &\geq& \frac{s}{N}S^{N/sp}+\frac{s}{N}\Iom |u_\infty|^{p^*_s}dx+\mu\left(\frac{1}{p}-\frac{1}{q+1}\right)
|\Om|^{\frac{p^*_s-(q+1)}{p^*_s}}\left(\Iom |u_\infty(x)|^{p^*_s}dx\right)^{\frac{q+1}{p^*_s}}\\
&:=& \frac{s}{N}S^{N/sp}+h(\eta),
\Eea
where $h(\eta)=\frac{s}{N}\eta^{p^*_s}+\mu\left(\frac{1}{p}-\frac{1}{q+1}\right)|\Om|^{\frac{p^*_s-(q+1)}{p^*_s}}\eta^{q+1}$
with $\eta=\displaystyle\left(\Iom |u_\infty(x)|^{p^*_s}dx\right)^{\frac{1}{p^*_s}}. $
By elementary analysis, we can show that $h$ attains its minimum at $\eta_0=\bigg(\frac{\mu(p-1-q)(N-sp)}{p^2s}\bigg)^{\frac{1}{p^*_s-(q+1)}}|\Om|^{\f{1}{p^*_s}}$
and
\Bea
h(\eta_0)&=&\f{s}{N}\bigg(\frac{\mu(p-1-q)(N-sp)}{p^2s}\bigg)^{\f{p^*_s}{p^*_s-(q+1)}}|\Om|\\
\qquad &-&\frac{\mu (p-1-q)}{p(q+1)}|\Om|^{\f{p^*_s-(q+1)}{p^*_s}}
\bigg(\frac{\mu(p-1-q)(N-sp)}{p^2s}\bigg)^{\f{q+1}{p^*_s-(q+1)}}|\Om|^{\f{q+1}{p^*_s}}\\
&=& -M\mu^{\frac{p^*_s}{p^*_s-(q+1)}},
\Eea
with $M$ given in \eqref{M}.
This in turn implies $c \geq \frac{s}{N}S^{\f{N}{sp}}-M\mu^{\f{p^*_s}{p^*_s-(q+1)}}$ and that gives a  contradiction to our hypothesis. Hence $b=0$.
This concludes that $u_k \to u_\infty$ strongly in $X_0.$
 \end{proof}

\begin{lemma}\label{inf-}
Let $N\in \N$ be such that $N> \f{sp}{2} [p+1+ \sqrt{(p+1)^2-4}]$ and $q\in(q_1, p-1)$, where \be\lab{feb-26-5}q_1:=\f{N^2(p-1)}{(N-sp)(N-s)}-1.\ee
 Then, there exists $\tilde{\mu}_1>0$ and $u_0\in X_0$ such that
 \begin{align}\label{J-mu}
 \sup_{t\geq0}I^+_\mu(tu_0)< \f{s}{N}S^{\f{N}{sp}}-M\mu^{\f{p^*_s}{p^*_s-q-1}},
 \end{align}
 for $\mu\in(0,\tilde{\mu}_1).$ In particular, 
\be\lab{feb-17-1}
 \tilde{\alpha}_\mu^-< \f{s}{N}S^{\f{N}{sp}}-M\mu^{\f{p^*_s}{p^*_s-q-1}}
\ee
where $I_{\mu}^+$ is defined as in \eqref{Feb-18-1} and $\al_{\mu}^-$ and $M$ are given as in \eqref{alpha-mu} and (\ref{M}) respectively.
\end{lemma}
 
\begin{proof}

Let $u_{\eps}$ be as defined in \eqref{u-eps}. Then we claim  \begin{align}\label{i2}
|u_\eps^+|_{L^{p^*_s}}=  |u_\eps|_{L^{p^*_s}}^{p^*_s} \geq S^{\f{N}{sp}} + o(\eps^{\f{N}{p-1}}).
 \end{align}
To see this, 
\bea\lab{Mar-5-1}
|u_\eps|_{L^{p^*_s}(\Om)}^{p^*_s}&=& \Iom |u_\eps|^{p^*_s}dx \geq \int_{\Om_\de} |u_\eps|^{p^*_s}dx\no\\
&=& \int_{\Om_\de}|U_\eps(x)|^{p^*_s}dx\no\\
&=& \int_{\Rn}|U_\eps(x)|^{p^*_s}dx-\int_{\Rn \setminus {\Om_\de}}|U_\eps(x)|^{p^*_s}dx.
\eea
Moreover, \Bea
\int_{\Rn \setminus {\Om_\de}}|U_\eps(x)|^{p^*_s}dx\leq \int_{\Rn \setminus B(0,R')}|U_\eps(x)|^{p^*_s}dx &=&  \f{1}{\eps^N}\int_{\Rn \setminus B(0,R')}U^{p^*_s}(\f{x}{\eps})dx\\
&\leq&C\int_{\f{R'}{\eps}}^{\infty}r^{N-1-\f{Np}{p-1}}dr\\
&\leq& C\eps^{\f{N}{p-1}}.
\Eea
Therefore substituting back to \eqref{Mar-5-1} we obtain 
$$|u_\eps|_{L^{p^*_s}(\Om)}^{p^*_s}\geq S^{\f{N}{sp}}-C\eps^{\f{N}{p-1}}.$$

Furthermore, a similar analysis as in \cite[Proposition 21]{SerVal} (see also \cite[Lemma 2.7]{MPSY}) yields, 
 for $\eps>0$ small $(0<\eps<\f{\de}{2})$ we have, 
 \begin{align}\label{i1}
  ||{u_\eps}||_{X_0}^p \leq S^{\f{N}{sp}} + o(\eps^{\f{N-ps}{p-1}}).
 \end{align}

Define, $$J(u):=\f{1}{p}||{u}||_{X_0}^p-\f{1}{p^*_s}|u^+|_{L^{p^*_s}}^{p^*_s}, \quad u \in X_0$$ and choose $\eps_0>0$ small such that 
\eqref{i1} and \eqref{i2} hold and Lemma \ref{l:4ii} is satisfied.  Let $\eps\in(0,\eps_0)$. Then, consider corresponding $u_0:=u_{\eps_0}$.
Let us consider the function $h: [0,\infty) \to \R$ defined by $h(t)=J(tu_0)$ for all $t \geq 0.$ It can be shown that $h$ attains
its maximum at $t=t_*=\bigg(\f{||{u_0}||_{X_0}^p}{|u_0^+|_{L^{p^*_s}}^{p^*_s}}\bigg)^{\f{1}{p^*-p}}$ and 
$\sup_{t \geq 0}J(tu_0)=\f{s}{N}\bigg(\f{||{u_0}||_{X_0}^p}{|u_0^+|_{L^{p^*_s}}^{p}}\bigg)^{\f{N}{sp}}.$ Using \eqref{i1} and \eqref{i2} a straight forward computation yields,

\begin{align}\label{i3}
 \sup_{t \geq 0}J(tu_0) \leq \f{s}{N}S^{\f{N}{sp}}+ o(\eps^{\f{N-sp}{p-1}}).
\end{align}
Since $I^+_{\mu}(tu_0)<0$ for $t$ small, we can find $t_0 \in (0,1)$ such that 
$$\sup_{0 \leq t \leq t_0}I^+_\mu(tu_0) \leq \f{s}{N}S^{\f{N}{sp}}-M\mu^{\f{p^*_s}{p^*_s-q-1}} ,$$ for $\mu>0$ small.
Hence, we are left to estimate $\sup_{t_0 \leq t} I^+_\mu(tu_0)$. 
\Bea
\sup_{t \geq t_0}I^+_\mu(tu_0)&=&\sup_{t \geq t_0}[J(tu_0)-\f{t^{q+1}}{q+1}|u_0^+|_{L^{q+1}}^{q+1}]\\
&\leq& \f{s}{N}S^{\f{N}{sp}}+o(\eps^{\f{N-sp}{p-1}})-\f{t^{q+1}}{q+1}|u_0|_{L^{q+1}}^{q+1}\\
&\leq&\begin{cases}
       \f{s}{N}S^{\f{N}{sp}}+c_1 \eps^{\f{N-ps}{p-1}}-c_2\mu \eps^{\f{(N-ps)(q+1)}{p(p-1)}}, \quad 0<q<\f{N(p-2)+ps}{N-sp}\\
       \f{s}{N}S^{\f{N}{sp}}+c_1 \eps^{\f{N-ps}{p-1}}-c_2\mu \eps^{\f{N}{p}}|ln \eps|, \,\,\ \quad q=\f{N(p-2)+ps}{N-sp}\\
       \f{s}{N}S^{\f{N}{sp}}+c_1 \eps^{\f{N-ps}{p-1}}-c_2\mu \eps^{N-\f{(N-sp)(q+1)}{p}}, \quad\f{N(p-2)+ps}{N-sp}<q<p-1.
      \end{cases}
\Eea
Choose $\eps\in (0,\f{\de}{2})$ such that $\eps^\f{N-sp}{p-1}=\mu^{\f{p^*_s}{p^*_s-q-1}}$. Then for $\f{N(p-2)+ps}{N-sp}<q<p-1$, the term
$\f{s}{N}S^{\f{N}{sp}}+c_1 \eps^{\f{N-ps}{p-1}}-c_2\mu \eps^{N-\f{(N-sp)(q+1)}{p}}$ reduces to $\f{s}{N}S^{\f{N}{sp}}+c_1 \mu^{\f{p^*_s}{p^*_s-q-1}}-c_2\mu \bigg(\mu^{\f{p^*}{p^*-q-1}}\bigg)^{(N-\f{(N-sp)(q+1)}{p})(\f{p-1}{N-ps})}$. Now, note that we can make 
$$c_1 \mu^{\f{p^*_s}{p^*_s-q-1}}-c_2\mu \bigg(\mu^{\f{p^*}{p^*-q-1}}\bigg)^{(N-\f{(N-sp)(q+1)}{p})(\f{p-1}{N-ps})} <-M\mu^{\f{p^*_s}{p^*_s-q-1}},$$ for $\mu>0$ small if we further choose
$(\f{p^*_s}{p^*_s-q-1})(\f{p-1}{p})[\f{Np}{N-ps}-(q+1)] < \f{p^*_s}{p^*_s-q-1}-1 $ i.e., if 
$q+1> \f{N^2(p-1)}{(N-sp)(N-s)}$. This proves \eqref{J-mu}. It is easy to see that \eqref{feb-17-1} follows by combining \eqref{J-mu} along with Lemma \ref{N.mu-i-2} .

\end{proof}

\subsection{Sign changing critical points of $I_{\mu}$}

Define $$\mathcal{N}^{-}_{\mu,1}:=\{u\in N_{\mu}: u^{+}\in N_{\mu}^{-}\},$$
$$\mathcal{N}^{-}_{\mu, 2}:=\{u\in N_{\mu}: -u^{-}\in N_{\mu}^{-}\},$$
We set
\begin{align}\label{ba-1-2}
 \ba_1=\inf_{u \in \mathcal{N}_{\mu,1}^-}I_\mu(u) \quad\text{and}\quad \ba_2=\inf_{u \in \mathcal{N}_{\mu,2}^-}I_\mu(u).
\end{align}

\begin{theorem} \label{t:4ii}
Let $p \geq 2$,  $N> \f{sp}{2} [p+1+\sqrt {(p+1)^2-4 }]$ and $q_1<q<p-1$, where $q_1$ is defined as in \eqref{feb-26-5}.
Assume $0<\mu<\text{min}\{\tilde\mu, \tilde\mu_1, \mu_{*}, \mu_1\}$, where $\tilde\mu$, $\tilde \mu_1$ and $\mu_1$ are as in \eqref{mu'}, Lemma \ref{inf-} and  Lemma \ref{l:6-ii} respectively.   
$\mu_*$ is chosen such that $\tilde\al_{\mu}^-$ is achieved 
in $(0, \mu_*)$.  Let $\ba_1$, $\ba_2$, $\tilde\al_{\mu}^-$ be  defined as in \eqref{ba-1-2} and \eqref{alpha-mu} respectively. 
\begin{itemize}
\item[(i)] Let $\ba_1<\tilde\al_{\mu}^{-}$. Then, there exists a sign changing critical point $\tilde w_1$ of 
$I_\mu$ such that  $\tilde w_1\in \mathcal{N}_{\mu,1}^-$ and $I_\mu(\tilde w_1)=\ba_1$.  
\item[(ii)]
If $\ba_2< \tilde\al_{\mu}^{-}$, then there exists a sign changing critical point $\tilde w_2$ of $I_\mu$ such that 
 $\tilde w_2\in \mathcal{N}_{\mu,1}^-$ and $I_\mu(\tilde w_2)=\ba_2$.
 \end{itemize}
\end{theorem}
\begin{proof}
(i) Let $\ba_1< \tilde \al_{\mu}^{-}$. We prove the theorem in few steps.

\vspace{2mm}

{\bf Step 1:}  $\mathcal{N}_{\mu,1}^-$ and $\mathcal{N}_{\mu,2}^-$ are closed sets. \\
To see this, let $\{u_n\} \subset \mathcal{N}_{\mu,1}^-$ such that $u_n \to u$ in $X_0$.
It is easy to note that $|u_n|, |u|\in X_0$ and  $|u_n| \to |u|$ in $X_0$. This in turn implies 
$u_n^+ \to u^+$ in $X_0$ and  $L^\ga(\Rn)$ for $\ga \in [1,p^*_s]$ (by Sobolev inequality).
Since, $u_n \in \mathcal{N}_{\mu,1}^-$, we have  $u_n^+ \in N_{\mu}^-$. Therefore
 \be\label{4-10}
  ||{u_n^+}||_{X_0}^p-|u_n^+|^{p^*_s}_{L^{p^*_s}(\Om)}-\mu|u_n^+|^{q+1}_{L^{q+1}(\Om)}=0
  \ee
  and
\be \label{4-11} 
  (p-1-q)||{u_n^+}||_{X_0}^p-(p^*_s-q-1)|u_n^+|^{p^*_s}_{L^{p^*_s}(\Om)}<0 \,\,\forall\  n \geq 1. 
 \ee
Passing to the limit as $n\to\infty$, we obtain $u^+\in N_{\mu}$ and \\ 
$(p-1-q)||{u^+}||_{X_0}^p-(p^*_s-q-1)|u^+|^{[p^*_s}_{L^{p^*_s}(\Om)}\leq 0$.
But, from Lemma \ref{N.mu-ii}, we know $N_{\mu}^0=\emptyset$.
Therefore $u^+\in N_{\mu}^{-}$ and hence $\mathcal{N}_{\mu,1}^{-}$ is closed. 
Similarly it can be shown that $\mathcal{N}_{\mu, 2}^{-}$ is also closed. Hence step 1 follows.

\vspace{2mm}

By Ekeland Variational Principle there exists sequence $\{u_n\} \subset \mathcal{N}_{\mu,1}^-$ 
such that 
\be\lab{4-11'}
I_\mu(u_n) \to \ba_1 \quad\text{and}\quad I_\mu(z) \geq I_\mu(u_n)-\frac{1}{n}||{u_n-z}||_{X_0} \quad\forall\ z\in \mathcal{N}_{\mu,1}^-.
\ee

\vspace{2mm}

{\bf Step 2:}  $\{u_n\}$ is uniformly bounded in $X_0$. \\
To see this, we notice $u_n \in \mathcal{N}_{\mu,1}^-$ implies $u_n \in N_{\mu}$ and this in turn implies
$\<I'_\mu(u_n),u_n\>=0$, that is,
$$||{u_n}||_{X_0}^p=|u_n|^{p^*_s}_{L^{p^*_s}(\Om)}+\mu|u_n|^{q+1}_{L^{q+1}(\Om)}. $$
Since $I_\mu(u_n)\to\ba_1$, using the above equality in the expression of $I_\mu(u_n)$, we get, for $n$ large enough
\Bea
\frac{s}{N}||{u_n}||_{X_0}^p&\leq& \ba_1+1+\displaystyle\left(\frac{1}{q+1}-\frac{1}{p^*_s}\right)\mu|u_n|^{q+1}_{L^{q+1}(\Om)}\\
&\leq& C(1 + ||{u_n}||_{X_0}^{q+1}).
\Eea
As $p>q+1,$ the above implies $\{u_n\}$ is uniformly bounded in $X_0$.\\
We note that for any $u\in X_0$, we have
\bea\lab{4-13}
||{u}||_{X_0}^p&=&\int_{\R^{2N}}\frac{|u(x)-u(y)|^p}{|x-y|^{N+ps}}dxdy\no\\
&=&\int_{\R^{2N}}\frac{(|u(x)-u(y)|^2)^{\frac{p}{2}}}{|x-y|^{N+ps}}dxdy\no\\
&=&\int_{\R^{2N}}\frac{\bigg(\big|(u^+(x)-u^+(y)\big)-\big(u^-(x)-u^-(y)\big)|^2\bigg)^{\frac{p}{2}}}{|x-y|^{N+ps}}dxdy\no\\
&=&\int_{\R^{2N}}\frac{\bigg(\big(u^+(x)-u^+(y)\big)^2+\big(u^-(x)-u^-(y)\big)^2+2u^+(x)u^-(y)+2u^+(y)u^-(x)\bigg)^{\frac{p}{2}}}{|x-y|^{N+ps}}dxdy\no\\
&\geq&\int_{\R^{2N}}\frac{\bigg(\big(u^+(x)-u^+(y)\big)^2+\big(u^-(x)-u^-(y)\big)^2\bigg)^{\frac{p}{2}}}{|x-y|^{N+ps}}dxdy\no\\
&\geq&\int_{\R^{2N}}\frac{\bigg(\big(u^+(x)-u^+(y)\big)^2\bigg)^{\frac{p}{2}}}{|x-y|^{N+ps}}dxdy
+\int_{\R^{2N}}\frac{\bigg(\big(u^-(x)-u^-(y)\big)^2\bigg)^{\frac{p}{2}}}{|x-y|^{N+ps}}dxdy\no\\
&=& ||{u^+}||_{X_0}^p+||{u^-}||_{X_0}^p
\eea
By a simple calculation, it follows 
\be\lab{4-14}|u|^{p^*_s}_{L^{p^*_s}(\Om)}=|u^+|^{p^*_s}_{L^{p^*_s}(\Om)}+|u^-|^{p^*_s}_{L^{p^*_s}(\Om)}\quad\text{and}\quad
|u|^{q+1}_{L^{q+1}(\Om)}=|u^+|^{q+1}_{L^{q+1}(\Om)}+|u^-|^{q+1}_{L^{q+1}(\Om)}. \ee
Combining \eqref{4-13} and \eqref{4-14}, we obtain
\be\lab{4-14'}
I_\mu(u) \geq I_\mu(u^+)+I_\mu(u^-) \quad\forall\quad u\in X_0.
\ee 
{\bf Step 3:} There exists $b >0$ such that $||u_n^-||_{X_0} \geq b$ for all $n \geq 1.$ \\
Suppose the step is not true.  Then for each $k \geq 1,$ there exists $u_{n_k}$ such that
\begin{align}\label{4-12}
||u_{n_k}^-||_{X_0}<\frac{1}{k}\,\,\, \forall\ k \ \geq 1.
\end{align}
Therefore, $||u_{n_k}^-||_{X_0} \to 0$ as $k\to\infty$ and by Sobolev inequality 
$$|u_{n_k}^-|_{L^{p^*_s}(\Om)} \to 0,\,\,\,|u_{n_k}^-|_{L^{q+1}(\Om)} \to 0, \quad\text{as}\quad k \to \infty.$$
Consequently, $I_\mu(u_{n_k}^-) \to 0$ as $k \to \infty$.  As a result, using \eqref{4-14'} we have
$$
\ba_1=I_\mu(u_{n_k})+o(1)\geq I_\mu(u^+_{n_k})+I_\mu(u^-_{n_k})+o(1)=I_{\mu}^+(u^+_{n_k})+o(1)\geq \tilde\al_{\mu}^{-}+o(1).
$$
This is a contradiction to the hypothesis. Hence step 3 follows.

\vspace{2mm}

{\bf Step 4}: $I'_\mu(u_n) \to 0$ in $(X_0)'$ as $n \to \infty$.

Since $u_n\in \mathcal{N}^{-}_{\mu, 1}$, we have $u_n^+\in N_{\mu}^-$. Thus by Lemma \ref{N.mu-iii} applied to the element $u_{n}^+$, there exists 
\be\lab{g_n}
\rho_n :=\rho_{u_n^+} \quad\text{and}\quad g_n:=g_{\rho_{u_n^+}},
\ee such that 
\be\lab{4-15}
g_n(0)=1, \quad \big(g_n(w)\big)(u_n^++w) \in N_{\mu}^- \quad\forall\quad w \in B_{\rho_n}(0).
\ee  Choose $0<\tilde\rho_n<\rho_n$ such that $\tilde\rho_n \to 0$.  Let $v\in X_0$ with $||v||_{X_0}=1$.
Define $$v_n := -\tilde\rho_n [v^+\chi_{\{u_n \geq 0\}}-v^-\chi_{\{u_n \leq 0\}}]$$
and 
\Bea
z_{\tilde\rho_n}&:=&\big(g_n(v_n^-)\big)(u_n-v_n)\\
&=:&z_{\tilde\rho_n}^1-z_{\tilde\rho_n}^2,
\Eea
where $z_{\tilde\rho_n}^1:=\big(g_n(v_n^-)\big)(u_n^{+} +\tilde\rho_n v^+\chi_{\{u_n \geq 0\}})$ and 
$z_{\tilde\rho_n}^2:=\big(g_n(v_n^-)\big)(u_n^{-} +\tilde\rho_n v^-\chi_{\{u_n \leq 0\}}).$
Note that $v_n^-=\tilde\rho_n v^+\chi_{\{u_n \geq 0\}}.$ So, $||v_n^-||_{X_0} \leq \tilde\rho_n ||v||_{X_0}\leq \tilde\rho_n.$ Hence taking $w=v_n^-$ in \eqref{4-15}
 we have, $z_{\tilde\rho_n}^+=  z_{\tilde\rho_n}^1 \in N_{\mu}^{-}$ so $z_{\tilde\rho_n} \in N_{\mu, 1}^-.$
Hence, 
$$I_\mu(z_{\tilde\rho_n}) \geq I_\mu(u_n)-\frac{1}{n} ||u_n-z_{\tilde\rho_n}||_{X_0}. $$

This implies,
\bea\label{22.02.E1}
\frac{1}{n}||u_n-z_{\tilde\rho_n}||_{X_0}&\geq& I_\mu(u_n)-I_\mu(z_{\tilde\rho_n})\notag\\
&=&\<I'_\mu(u_n),u_n-z_{\tilde\rho_n}\>+o(1) ||u_n-z_{\tilde\rho_n}||_{X_0}\notag\\
&=&-\<I'_\mu(u_n),z_{\tilde\rho_n}\>+o(1) ||u_n-z_{\tilde\rho_n}||_{X_0},
\eea
as $\<I'_\mu(u_n),u_n\>=0 $ for all $n.$
Let $w_n=\tilde\rho_n v.$ Then,
\begin{align}\label{D1}
\frac{1}{n} ||u_n-z_{\tilde\rho_n}||_{X_0} \geq -\<I'_\mu(u_n),w_n+z_{\tilde\rho_n}\>+\<I'_\mu(u_n),w_n\>\notag\\
             +o(1) ||u_n-z_{\tilde\rho_n}||_{X_0}.
\end{align}
Now, $\<I'_\mu(u_n),w_n\>=\<I'_\mu(u_n),\tilde\rho_n v\>=\tilde\rho_n\<I'_\mu(u_n), v\>.$
Define $$\overline{v_n}:=v^+\chi_{\{u_n \geq 0\}}-v^-\chi_{\{u_n \leq 0\}}.$$

So, $z_{\tilde\rho_n}=g_n(v_n^-)(u_n-\tilde\rho_n\overline{v_n}).$
Hence we have,
\begin{align}\label{D2}
\<I'_\mu(u_n), w_n+z_{\tilde\rho_n}\>= \<I'_\mu(u_n),w_n+g_n(v_n^-)(u_n-\tilde\rho_n\overline{v_n})\> =\<I'_\mu(u_n),\tilde\rho_n v-g_n(v_n^-)\tilde\rho_n\overline{v_n}\>\no\\
 =\tilde\rho_n\<I'_\mu(u_n),v-g_n(v_n^-)\overline{v_n}\>
\end{align}

Using \eqref{D2} in \eqref{D1}, we have
\begin{align}\label{22.02.E2}
\frac{1}{n} ||u_n-z_{\tilde\rho_n}||_{X_0} \geq -\tilde\rho_n\<I'_\mu(u_n),v-g_n(v_n^-)\overline{v_n}\>\notag\\
             +\tilde\rho_n\<I'_\mu(u_n),v\>+o(1) ||u_n-z_{\tilde\rho_n}||_{X_0}.
\end{align}

First we will estimate $\<I'_\mu(u_n),v-g_n(v_n^-)\overline{v_n}\>$. For this,
\Bea
v-g_n(v_n^-)\overline{v_n}&=&v^+-v^--g_n(v_n^-)[v^+\chi_{\{u_n \geq 0\}}-v^-\chi_{\{u_n \leq 0\}}]\\
&=&v^+[g_n(0)-g_n(v_n^-)\chi_{\{u_n \geq 0\}}]-v^-[g_n(0)-g_n(v_n^-)\chi_{\{u_n \leq 0\}}]\\
&=&- v^+[\<g_n'(0),v_n^-\>+o(1) ||v_n^-||_{X_0}]+v^-[\<g_n'(0),v_n^-\>+o(1)||{v_n^-}||_{X_0}]\\
&=&- v^+\tilde\rho_n[\<g_n'(0),v^+\>+o(1)||{v^+}||_{X_0}]+v^-\tilde\rho_n[\<g_n'(0),v^+\>+o(1) ||{v^+}||_{X_0}]\\
&=&-\tilde\rho_n\big[\<g_n'(0),v^+\>+o(1) ||{v^+}||_{X_0}\big]v.
\Eea
Therefore,
\begin{align}\lab{Mar-5-2}
 \<I'_\mu(u_n),v-g_n(v_n^-)\overline{v_n}\> 
     =-\tilde\rho_n\big(\<g_n'(0),v^+\>+o(1)\norm{v^+}\big)\<I'_\mu(u_n),v\>.
\end{align}

{\bf Claim :}  $g_n(v_n^-)$ is uniformly bounded in $X_0.$ 

To see this, we observe that  from \eqref{4-15} we have, $g_n(v_n^-)(u_n^++v_n^-) \in N^-_\mu \subset N_\mu,$ which implies,
$$||{c_n\tilde\psi_n}||_{X_0}^p-\mu|c_n\tilde\psi_n|^{q+1}_{L^{q+1}(\Om)}-|c_n\tilde\psi_n|_{L^{p^*_s}(\Om)}^{p^*_s}=0,$$
where $c_n:=g_n(v_n^-)$ and $\tilde\psi_n:=u_n^+ + v_n^-.$
Dividing by $c_n^{p^*}$ we have,
\begin{align}\label{D4}
 c_n^{p-p^*}||{\tilde\psi_n}||_{X_0}^p-\mu c_n^{q+1-p^*}|\tilde\psi_n|^{q+1}_{L^{q+1}(\Om)}=|\tilde\psi_n|_{L^{p^*_s}(\Om)}^{p^*_s}.
\end{align}
Note that $||\tilde\psi_n||_{X_0}$ is uniformly bounded above as $||u_n||_{X_0}$ is uniformly bounded and $\tilde\rho_n=o(1)$. Also,
$||{\tilde\psi_n}||_{X_0} \geq ||{u_n^{+}}||_{X_0}-\tilde\rho_n ||{v}||_{X_0}$. Note that $||{u_n^{+}}||_{X_0} \geq \tilde b$ for large $n$.
If not, then $||{u_n^+}||_{X_0} \to 0$ as $n \to \infty.$ As $u_n \in N_{\mu,1}^-,$ so $u_n^+ \in N_\mu^-.$
Now, $N^{-}_{\mu}$ is a closed set and $0 \notin N^{-}_{\mu}$ and therefore $||{u_n^{-}}||_{X_0} \not\to 0$ as $n \to \infty.$  
Thus there exists $\tilde b \geq 0$ such that $||{u_n^+}||_{X_0} \geq \tilde b>0$.  This in turn implies that 
$||\tilde\psi_n||_{X_0} \geq C$, for some $C>0$ by choosing $\tilde\rho_n$ small enough. Consequently, if $c_n$ is not uniformly bounded,
we obtain LHS of \eqref{D4} converges to $0$ as $n\to\infty $. 

On the other hand,  $$|\tilde\psi_n|_{L^{p^*_s}(\Om)} \geq|u_n^{+}|_{L^{p^*_s}(\Om)}-\tilde\rho_n|v|_{L^{p^*_s}(\Om)}>c,$$ for some positive constant $c$ as $\rho_n=o(1)$ and $u_n^+\in N_{\mu}^-$ implies 
$$(p^*_s-1-q)|u_n^+|_{L^{p^*_s}(\Om)}^{p^*_s}>(p-1-q)||u_n^+||_{X_0}^p>(p-1-q)\tilde b^p.$$ 
Hence, the claim follows.

\vspace{3mm}

Now using the fact that $g_n(0)=1$ and the above claim we obtain
\Bea
||{u_n-z_{\tilde\rho_n}}||_{X_0}&\leq& ||{u_n}||_{X_0}\big|1-g_n(v_n^-)\big|+\tilde\rho_n ||{\overline{v_n}}||_{X_0}g_n(v_n^-)\\
&\leq& ||{u_n}||_{X_0}\big[|\<g_n'(0),v_n^-\>|+o(1) ||{\overline{v_n}||_{X_0}}\big]+\tilde\rho_n||v||_{X_0}g_n(v_n^-)\\
&\leq& \tilde\rho_n \big[||{u_n}||_{X_0}\<g_n'(0),\overline{v_n}^+\>+o(1)||{v}||_{X_0}+||{v}||_{X_0}g_n(v_n^-)\big]\\
&\leq& \tilde\rho_n C.
\Eea
Substituting this and \eqref{Mar-5-2} in \eqref{22.02.E2} yields
$$\tilde\rho_n\bigg(\<g_n'(0),v^+\>+o(1) ||{v^+}||_{X_0}\bigg)\<I'_\mu(u_n),v\>+  \<I'_\mu(u_n),v\>\tilde\rho_n+\tilde\rho_n o(1) \leq \tilde\rho_n.\f{C}{n}.$$
This implies
$$\bigg[\big(\<g_n'(0),v^+\>+o(1) ||{v^+}||_{X_0}\big)+1\bigg]    \<I'_\mu(u_n),v\> \leq \f{C}{n}+o(1)\quad\mbox{for all}\quad n \geq n_0.$$ Since $|\<g_n'(0),v^+\>|$ is uniformly bounded (see Lemma \ref{l:6-ii} in Appendix) , letting $n \to \infty$ we have
$I'_\mu(u_n) \to 0$ in $(X_0)'.$ Hence the step 4 follows.

\vspace{5mm}

Therefore $\{u_n\}$ is a (PS) sequence of $I_{\mu}$ at level $\ba_1< \tilde\al_{\mu}^{-} $.
From lemma \ref{inf-}, it follows that 
$$\tilde\al_{\mu}^{-}<\frac{s}{N}S^{\frac{N}{ps}}-M\mu^{\frac{p^*_s}{p^*_s-q-1}}\quad\mbox{for}\quad \mu \in(0,\tilde{\mu}_1),  $$ 
where  $M=\f{\big(pN-(N-ps)(q+1)\big)(p-1-q)}{p^2(q+1)}\big(\f{(p-1-q)(N-ps)}{p^2s}\big)^\f{q+1}{p^*_s-q-1}|\Om|.$ 
Thus,
$$\ba_1<\tilde\al_{\mu}^{-}<\frac{s}{N}S^{\frac{N}{ps}}-M\mu^{\frac{p^*_s}{p^*_s-q-1}}.$$
On the other hand, it follows from the Lemma \ref{ps} that $I_\mu$ satisfies $PS$ at level $c$ for
$$c<\frac{s}{N}S^{\frac{N}{ps}}-M\mu^{\frac{p^*_s}{p^*_s-q-1}},$$

this yields, there exists $u\in X_0$ such that $u_n\to u$ in $X_0$. 
By doing a simple calculation we get $u_n^- \to u^-$  in $X_0$. 
Consequently, by Step 3  $||{u^-}||_{X_0} \geq b$. As $\mathcal{N}_{\mu,1}^-$ is a closed set and $u_n \to u$, we obtain $u \in 
\mathcal{N}_{\mu,1}^-$,  that is, $u^+ \in {N}_{\mu}^-$ and $u^+ \neq 0. $
Therefore $u$ is a solution of $(\mathcal{P}_{\mu})$ with $u^+$ and $u^-$ are both nonzero. Hence, $u$ is a sign-changing solution of $(\mathcal{P}_{\mu})$. Define $\tilde w_1:=u$. This completes the proof of part (i) of the theorem.

\vspace{2mm}

Proof of part (ii) is similar to part (i) and we omit the proof. 
\end{proof}

\begin{theorem}\lab{t:4i}
Let $\ba_1, \ba_2\geq  \tilde\al_{\mu}^{-}$ where $\ba_1$, $\ba_2$, $ \tilde\al_{\mu}^{-}$ be defined as in \eqref{ba-1-2} and \eqref{alpha-mu} respectively. Then,  there exists $\mu_0>0$ such that for any $\mu\in(0,\mu_0)$, $I_{\mu}$ has a sign changing critical point in the following cases:\\

(i) for $p\geq\f{3+\sqrt{5}}{2}$, there exists $q_2:=
 \f{Np}{N-sp}-\f{p}{p-1}$ such that when $q>q_2$ and $N>sp(p^2-p+1)$,
   
(ii) for $2\leq p<\f{3+\sqrt{5}}{2}$, there exists $q_3:= \f{N(p-1)}{N-sp}-\f{p-1}{p}$ such that when $q>q_3$ and $N>sp(p+1)$.
 \end{theorem}

We need the following Proposition to prove the above Theorem \ref{t:4i}.
\begin{proposition}\lab{p:limit}
 Assume $0<\mu<\text{min}\{\mu_{*}, \tilde\mu, \tilde\mu_1\}$, where 
$\tilde\mu$ is as defined in $\eqref{mu'}$ and $\mu_{*}>0$ is chosen such that 
$\tilde\al_\mu^{-}$  is achieved in $(0, \mu_{*})$ and $\tilde{\mu_1}$ is as in Lemma \ref{inf-}. Then, for $p\geq\f{3+\sqrt{5}}{2}$, there exists $q_2:=
 \f{Np}{N-sp}-\f{p}{p-1}$ such that when $q>q_2$ and $N>sp(p^2-p+1)$ 
we have $$\sup_{a\geq 0,\ b\in\R} I_{\mu}(a w_1-bu_{\eps})< \tilde\al_{\mu}^{-}+\f{s}{N}S^\f{N}{ps},$$ for $\eps>0$ sufficiently small , 
where $w_1$  is a positive solution of $(\mathcal{P_\mu})$ and $u_{\eps}$ be as in \eqref{u-eps}. 

Furthermore, when $2\leq p<\f{3+\sqrt{5}}{2}$, there exists $q_3:= \f{N(p-1)}{N-sp}-\f{p-1}{p}$ such that
when $q>q_3$ and $N>sp(p+1)$,  it holds
$$\sup_{a\geq 0,\ b\in\R} I_{\mu}(a w_1-bu_{\eps})< \tilde\al_{\mu}^{-}+\f{s}{N}S^\f{N}{ps},$$ for $\eps>0$ sufficiently small .
\end{proposition}

To prove the above proposition, we need the following lemmas.
\begin{lemma}\lab{l:4vi}
Let  $w_1$ and $\mu$ be as in Proposition \ref{p:limit}. Then $$\sup_{s>0}I_{\mu}(s w_1)= \tilde\al_{\mu}^{-}.$$
\end{lemma} 
\begin{proof}
By the definition of   $\tilde\al_{\mu}^{-}$, we have  $\tilde\al_{\mu}^{-}=\inf_{u \in N_{\mu}^-}I^+_{\mu}(u)= I^+_{\mu}(w_1)=I_{\mu}(w_1)$. In the last equality we have used the fact that $w_1>0$.  Define $g(s):=I_{\mu}(sw_1)$. From the proof of Lemma \ref{N.mu-i}, it follows that there exists only two critical points of $g$, namely $t^{+}(w_1)$ and $t^{-}(w_1)$ and $\max_{s>0}g(s)=g(t^{+}(w_1))$. On the other hand $\<{I'}_{\mu}(w_1), v \>=0$ for every $v\in X_0$. Therefore $g'(1)=0$ which implies either $t^{+}(w_1)=1$ or $t^{-}(w_1)=1$. \\
{\it Claim}: $t^{-}(w_1)\not=1$.\\
To see this, we note that $t^{-}(w_1)=1$ implies $t^{-}(w_1)w_1\in N_{\mu}^{-}$ as $w_1\in N_{\mu}^{-}$. Using Lemma \ref{N.mu-i}, we know $t^{-}(w_1)w_1\in N_{\mu}^{+}$. Thus $N_{\mu}^{+}\cap N_{\mu}^{-}\not=\emptyset$, which is a contradiction. Hence we have the claim.\\
Therefore $t^{+}(w_1)=1$ and this completes the proof.
\end{proof}
\begin{lemma}\lab{l:4vii}
Let $u_{\eps}$ be as in \eqref{u-eps} and  $\mu$ be as in Proposition \ref{p:limit}. 
Then for $\eps>0$ sufficiently small, we have 
 $$\sup_{t\in\R}I_{\mu}(t u_{\eps})=\f{s}{N}S^\f{N}{ps}+
C\eps^{\f{(N-ps)}{(p-1)}} - k_8|u_{\eps}|^{q+1}_{L^{q+1}(\Om)}.$$
\end{lemma}
\begin{proof}
Define $\tilde\phi(t)=\f{t^p}{p}||u_{\eps}||_{X_0}^p-\f{t^{p^*_s}}{p^*_s}|u_{\eps}|_{L^{p^*_s}(\Om)}^{p^*_s}$. 
Thus $I_{\mu}(t u_{\eps})=\tilde\phi(t)-\mu\f{t^{q+1}}{q+1}|u_{\eps}|^{q+1}_{L^{q+1}(\Om)}$.
On the other hand, applying the analysis done in Lemma \ref{N.mu-i} to $u_{\eps}$, 
we obtain there exists $(t_0)_{\eps}= 
\bigg(\frac{(p-1-q)||{u_{\eps}}||_{X_0}^p}{(p^*_s-1-q)|u_{\eps}|^{p^*_s}_{L^{p^*_s}(\Om)}}\bigg)^{\frac{N-ps}{p^2s}}<t^{+}_{\eps}$ such that 
\Bea
\sup_{t\in\R}I_{\mu}(t u_{\eps})= \sup_{t\geq 0}I_{\mu}(t u_{\eps})=I_{\mu}(t^{+}_{\eps} u_{\eps})&=& \tilde\phi(t^{+}_{\eps})-\mu\f{{(t^{+}_{\eps})}^{q+1}}{q+1}|u_{\eps}|^{q+1}_{L^{q+1}(\Om)}\\
&\leq& \sup_{t\geq 0}\tilde\phi(t)-\mu\f{(t_0)_{\eps}^{q+1}}{q+1}|u_{\eps}|^{q+1}_{L^{q+1}(\Om)}.
\Eea
Substituting the value of $(t_0)_{\eps}$ and using Sobolev inequality, 
we have $$\mu\f{(t_0)_{\eps}^{q+1}}{q+1}\geq \f{\mu}{q+1}\bigg(\f{p-1-q}{p^*_s-q-1}S\bigg)^\f{(N-ps)(q+1)}{p^2s}=k_8.$$
Consequently, 
\be\lab{4-8}
\sup_{t\in\R}I_{\mu}(t u_{\eps})\leq \sup_{t\geq 0}\tilde\phi(t)-k_8|u_{\eps}|^{q+1}_{L^{q+1}(\Om)}.
\ee
Using elementary analysis, it is easy to check that $\tilde\phi$ attains it's maximum at the point
$\tilde t_0= \bigg(\f{||u_{\eps}||_{X_0}^p}{|u_{\eps}|^{p^*_s}_{L^{p^*_s}(\Om)}}\bigg)^\f{1}{p^*_s-p}$ 
and $\tilde\phi(t_0)=\f{s}{N}\bigg(\f{||u_{\eps}||_{X_0}^p}{|u_{\eps}|^{p}_{L^{p^*_s}(\Om)}}\bigg)^\f{N}{ps}$.\\
Moreover, using \eqref{i1} and \eqref{i2}, we can deduce as in \eqref{i3} that \be\lab{feb-18-4}\tilde\phi(t_0) \leq \f{s}{N}S^\f{N}{ps}+C\eps^{\f{(N-ps)}{(p-1)}}.\ee

Substituting back \eqref{feb-18-4} into \eqref{4-8}, completes the proof.
\end{proof}

\vspace{2mm}

{\bf Proof of Proposition \ref{p:limit}}: Note that, for fixed $a$ and $b$, $I_{\mu}\big(\eta(a w_1-bu_{\eps, \de})\big)\to -\infty$ as $|\eta|\to\infty$. Therefore $\sup_{a\geq 0,\ b\in\R} I_{\mu}(a w_1-bu_{\eps, \de})$ exists and supremum will be attained in $a^2+b^2\leq R^2$, for some large $R>0$. Thus it is enough to estimate  $I_{\mu}(a w_1-b u_{\eps, \de})$ in  $\{(a,b)\in \R^{+}\times\R: a^2+b^2\leq R^2 \}$. Using elementary inequality, there exists $d(m)>0$ such that 
\be\lab{in:ele}
|a+b|^m\geq |a|^m+|b|^m-d(|a|^{m-1}|b|+|a||b|^{m-1}) \quad\forall\quad a,\  b\in\R, \ m>1.
\ee 
Define, $f(v):=||v||_{X_0}^p$. Then using Taylor's theorem
\Bea f(a w_1-bu_{\eps, \de})&=&f(aw_1)-\langle f'(aw_1), bu_{\eps} \rangle+o(||bu_{\eps, \de}||_{X_0}^2)\no\\
&\leq&||a w_1||_{X_0}^p \no\\
&-& p
\int_{\R^{2N}}\frac{|aw_1(x)-aw_1(y)|^{p-2}(aw_1(x)-aw_1(y))(bu_{\eps, \de}(x)-bu_{\eps, \de}(y))}{|x-y|^{N+ps}}dxdy\\
&+&c||bu_{\eps, \de}||_{X_0}^2,\\
\Eea
where $c>0$ is small enough. We also note that from the definition of $u_{\eps,\de}$, it follows that $||u_{\eps, \de}||_{X_0}$ is bounded away from $0$. Therefore, since $p\geq 2$ we have $c||bu_{\eps, \de}||_{X_0}^2\leq ||bu_{\eps, \de}||_{X_0}^p$, for $c>0$ small enough. Hence
\Bea 
||a w_1-bu_{\eps, \de}||^p_{X_0}&=& ||a w_1||^p_{X_0} \no\\
&-& p
\int_{\R^{2N}}\frac{|aw_1(x)-aw_1(y)|^{p-2}(aw_1(x)-aw_1(y))(bu_{\eps, \de}(x)-bu_{\eps, \de}(y))}{|x-y|^{N+ps}}dxdy\\
&+&||b u_{\eps, \de}||^p_{X_0}
\Eea
Consequently,  $a^2+b^2\leq R^2$ implies
\Bea
I_{\mu}(a w_1-bu_{\eps, \de})&\leq& \f{1}{p}||a w_1||^p_{X_0}\\
&-& \int_{\R^{2N}}\frac{|aw_1(x)-aw_1(y)|^{p-2}(aw_1(x)-aw_1(y))(bu_{\eps, \de}(x)-bu_{\eps, \de}(y))}{|x-y|^{N+ps}}dxdy \\
&+&\f{1}{p}||b u_{\eps, \de}||^p_{X_0}-\f{1}{p_s^*}
\Iom|a w_1|^{p_s^*}dx-\f{1}{p_s^*}\Iom|b u_{\eps, \de}|^{p_s^*}dx\\
&-&\f{\mu}{q+1}\Iom|a w_1|^{q+1}dx-\f{\mu}{q+1}\Iom|b u_{\eps, \de}|^{q+1}dx\\
&+&C\displaystyle\left(\Iom|aw_1|^{p_s^*-1}|b u_{\eps, \de}|dx+\Iom|a w_1||b u_{\eps, \de}|^{p_s^{*}-1}dx\right)\\
&+&C\displaystyle\left(\Iom|aw_1|^{q}|b u_{\eps, \de}|dx+\Iom|a w_1||b u_{\eps, \de}|^{q}dx\right)\\
&=& I_{\mu}(a w_1)+I_{\mu}(b u_{\eps, \de})-a^{q}b\mu\Iom |w_1|^{q-1}w_1u_{\eps, \de}dx\\
&-&a^{p_s^*}b\Iom |w_1|^{p_s^*-2}w_1u_{\eps, \de}dx\\
&+&C\displaystyle\left(\Iom|w_1|^{p_s^*-1}| u_{\eps, \de}|dx+\Iom| w_1|| u_{\eps, \de}|^{p_s^{*}-1}dx\right)\\
&+&C\displaystyle\left(\Iom|w_1|^{q}| u_{\eps, \de}|dx+\Iom| w_1|| u_{\eps, \de}|^{q}dx\right).
\Eea

 Using Lemmas \ref{l:4i}, \ref{l:4vi} and \ref{l:4vii} we estimate in $a^2+b^2\leq R^2$,
 \be
 I_{\mu}(a w_1-bu_{\eps, \de})\leq \tilde\al_{\mu}^{-}
 +\f{s}{N}S_s^\f{N}{ps} - k_8|u_{\eps}|^{q+1}_{L^{q+1}(\Om)}
 + C\left(\eps^{\f{(N-ps)}{(p-1)}}+\eps^\f{N-ps}{p(p-1)}+\eps^\f{(N-ps)q}{p(p-1)}+\eps^\f{N(p-1)+ps}{p(p-1)}\right)\no.
 \ee
For the term $k_8|u_{\eps}|^{q+1}_{L^{q+1}(\Om)}$, we invoke  Lemma \ref{l:4ii}. Therefore when $\f{N(p-2)+ps}{N-ps}<q<p-1$, we have
\bea\lab{feb-26-4}
I_{\mu}(a w_1-bu_{\eps, \de})\leq   \tilde\al_{\mu}^{-}
&+&\f{s}{N}S_s^\f{N}{ps} - k_9\eps^{N-\f{(N-ps)(q+1)}{p}}\no\\
&+& C\left(\eps^{\f{(N-ps)}{(p-1)}}+\eps^\f{N-ps}{p(p-1)}+\eps^\f{(N-ps)q}{p(p-1)}+\eps^\f{N(p-1)+ps}{p(p-1)}\right)
\eea
 We will choose $q$ in such a way that the term $k_9\eps^{N-\f{(N-ps)(q+1)}{p}}$ dominates the other term 
involving $\eps$. Note that among the terms in the bracket, $\eps^\f{N-ps}{p(p-1)}$ and $\eps^\f{(N-ps)q}{p(p-1)}$ dominate the others.
  
This in turn implies we have to choose $q$ such that 
\be\lab{feb-26-1}
{N-\f{(N-ps)(q+1)}{p}}< \f{N-ps}{p(p-1)}
\ee
and
\be\lab{feb-26-2}
{N-\f{(N-ps)(q+1)}{p}}< \f{(N-ps)q}{p(p-1)}.
\ee
\eqref{feb-26-1} and \eqref{feb-26-2} implies $q>q_2$ and $q>q_3$ respectively, where 
\be\lab{feb-26-3}
q_2:= \f{Np}{N-sp}-\f{p}{p-1} \quad\text{and}\quad q_3:= \f{N(p-1)}{N-sp}-\f{p-1}{p}.
\ee
{\bf Case 1:} $p\geq \f{3+\sqrt{5}}{2}$\\
In this case by straight forward calculation it follows that $q_2>q_3$. So in this case, we choose $q>q_2$. Moreover, since $q<p-1$, to make the interval $(q_2, p-1)\not=\emptyset$, we have to take $N>sp(p^2-p+1)$.

{\bf Case 2:} $2\leq p< \f{3+\sqrt{5}}{2}$\\
In this case again by simple calculation it follows that $q_3>q_2$. Thus, in this case, we choose $q>q_3$. Furthermore, as  $q<p-1$, to make the interval $(q_3, p-1)\not=\emptyset$, we have to take $N>sp(p+1)$.

\vspace{3mm}

 Hence in both the cases  taking $\eps>0$ to be small enough in \eqref{feb-26-4}, we obtain  
$$\sup_{a\geq 0, b\in\R} I_{\mu}(a w_1-bu_{\eps, \de})<   \tilde\al_{\mu}^{-}+\f{s}{N}S_s^\f{N}{ps}.$$
\hfill{$\square$}

\vspace{4mm}

{\bf Proof of Theorem \ref{t:4i}}: Define $\mu_0:=\min\{\tilde\mu, \mu_{*} \}$,
\be\lab{eq:N-star}
\mathcal{N}^{-}_*:=\mathcal{N}^{-}_{\mu,1}\cap \mathcal{N}^{-}_{\mu,2}.
\ee and 
\be\lab{c-2} c_2:= \inf_{u\in \mathcal{N}^{-}_* } I_{\mu}(u),\ee 

Let $\mu\in (0, \mu_0)$. Using Ekland's variational principle and similar to the proof of Theorem \ref{t:4ii}, we obtain a  sequence $\{u_n\}\in \mathcal{N}^{-}_{*}$ satisfying 
$$I_{\mu}(u_n)\to c_2, \quad I'_{\mu}(u_n)\to 0 \quad\text{in}\quad (X_0)'.$$ 
Thus $\{u_n\}$ is a (PS) sequence at level $c_2$. From Lemma \ref{l:6-i}, given below, it follows that there exists $a>0$ and $b\in R$ 
such that $a w_1-bu_{\eps}\in \mathcal{N}^{-}_{*}$. Therefore Proposition \ref{p:limit} yields 
\be\lab{c2c1}
c_2< \tilde\al_{\mu}^{-}+\f{s}{N}S^\f{N}{ps}.
\ee

{\bf Claim 1:} There exists two positive constants $c, C$ such that 
$ 0< c \leq ||{u_n^{\pm}}||_{X_0} \leq C$. \\
To see this, we note that $\{u_n\} \subset \mathcal{N}_*^{-}\subset \mathcal{N}_{\mu,1}^{-}$.  Thus using \eqref{4-13}, Step 2 and Step 3
of the proof of Theorem \ref{t:4ii},  we have $||{u_n^{\pm}}||_{X_0} \leq C$ and $||{u_n^{-}}||_{X_0}\geq c$. To show  $||{u_n^{+}}||_{X_0}\geq a$
for some $a>0$, we use method of contradiction. Assume up to a subsequence  $||{u_n^{+}}||_{X_0} \to 0$ as $n \to \infty$. 
This  together with Sobolev embedding implies $|u_n^+|_{L^{p^*_s}(\Om)} \to 0$. On the other hand,
$u_n^{+} \in N_{\mu}^{-}$ implies $(p-1-q) ||{u_n^{+}}||_{X_0}^p-(p^*_s-q-1)|u_n^+|^{p^*_s}_{L^{p^*_s}(\Om)}<0$. 
Therefore by Sobolev inequality, we have 
\begin{align} 
S\leq \frac{||{u_n^{+}}||_{X_0}^p}{|u_n^+|^{p}_{L^{p^*_s}(\Om)}}<\f{p^*_s-q-1}{p-1-q}|u_n^+|^{p^*_s-p}_{L^{p^*_s}(\Om)},\no
\end{align}
which is a contradiction to the fact that $|u_n^+|_{L^{p^*_s}(\Om)} \to 0$. Hence the claim follows.

\vspace{2mm}

Going to a subsequence if necessary we have
\be\lab{4-22}
u_n^{+} \deb \eta_1,\,\, u_n^{-} \deb \eta_2 \quad\mbox{in} \quad X_0. 
\ee
{\bf Claim 2}: $\eta_1 \nequiv 0,\,\,\eta_2 \nequiv 0. $\\
 Suppose not, that is $\eta_1 \equiv 0. $ Then by compact embedding, 
$u_n^{+} \to 0$ in $L^{q+1}(\Om)$. 
Moreover, $u_n^{+} \in N_{\mu}^{-}\subset N_{\mu}$, implies $\<I'_\mu(u_n^{+}),u_n^{+}\>=0$.  Consequently, 
$$||{u_n^{+}}||_{X_0}^p-|u_n^{+}|^{p^*_s}_{L^{p^*_s}(\Om)}=\mu|u_n^{+}|^{q+1}_{L^{q+1}(\Om)}=o(1).$$
So we have $|u_n^{+}|^{p^*_s}_{L^{p^*_s}(\Om)}=||{u_n^{+}}||_{X_0}^p+o(1). $ This together with $||{u_n^{+}}||_{X_0} \geq c$  implies 
$$\frac{|u_n^{+}|^{p^*_s}_{L^{p^*_s}(\Om)}}{||{u_n^{+}}||_{X_0}^p} \geq 1+ o(1). $$
This along with Sobolev embedding gives $|u_n^{+}|^{p^*_s}_{L^{p^*_s}(\Om)} \geq S^{N/ps}+o(1)$.
Thus we have, 
\be\lab{4-20}
 I_\mu(u_n^{+})= \frac{1}{p}||{u_n^{+}}||_{X_0}^p-\frac{1}{p^*_s}|u_n^{+}|^{p^*_s}_{L^{p^*_s}(\Om)}+o(1)
 \geq \frac{s}{N}S^{N/ps}+o(1).
 \ee
Moreover,  $u_n \in \mathcal{N}_*^{-}$ implies $-u_n^{-} \in N_{\mu}^{-}$. Therefore using the given condition on $\ba_2$, we get 
\be\lab{4-21}
I_\mu(-u_n^{-}) \geq \ba_2 \geq \tilde\al_{\mu}^{-}.
\ee
 Also it follows  \ $I_\mu(u_n^{+})+I_\mu(-u_n^{-}) \leq I_\mu(u_n)=c_2+o(1)$ (see \eqref{4-14'}). 
 Combining this along with \eqref{4-21} and \eqref{c2c1}, we obtain
$$I_\mu(u_n^{+}) \leq c_2-  \tilde\al_{\mu}^{-}+o(1)<\frac{s}{N}S_s^{N/ps},$$
which is a contradiction to \eqref{4-20}. Therefore $\eta_1 \neq 0.$ Similarly $\eta_2 \neq 0$ and this proves the claim.

\vspace{2mm}

Set $w_2 :=\eta_1-\eta_2.$

{\bf Claim 3:} $w_2^{+}=\eta_1$ and $w_2^{-}=\eta_2$ a.e..\\ 
To see the claim we observe that $\eta_1\eta_2=0$ a.e. in $\Om$. Indeed,
\bea\lab{4-22'}
\displaystyle |\Iom \eta_1\eta_2 dx| &=&|\Iom (u_n^+-\eta_1)u_n^- dx+\Iom \eta_1(u_n^--\eta_2)dx|\no\\ 
&\leq& |u_n^+-\eta_1|_{L^p(\Om)}|u_n^{-}|_{L^{p'}(\Om)}+|\eta_1|_{L^{p'}(\Om)}|u_n^{-} -\eta_2|_{L^p(\Om)}
\eea
where $\frac{1}{p}+\frac{1}{p'}=1.$
By compact embedding we have $u_n^{+}\to \eta_1$ and $u_n^{-}\to \eta_2$ in $L^p(\Om)$.
As $p \geq \f{2N}{N+s},$ then $p'\leq p^*_s.$ 
Therefore, using claim 1,
we pass the limit in \eqref{4-22'} and obtain $\Iom \eta_1\eta_2 dx=0$. Moreover  by \eqref{4-22}, $\eta_1, \ \eta_2\geq 0$ a.e.. Hence $\eta_1\eta_2=0$ a.e. in $\Om$. We have $w_2^+-w_2^-=w_2=\eta_1-\eta_2$. It is easy to check that $w_2^+\leq \eta_1$ and $w_2^{-}\leq \eta_2$.  To show that equality holds a.e. we apply method of contradiction. Suppose, there exists $E\subset\Om$ such that $|E|>0$ and $0\leq w_2^+(x)<\eta_1(x) \ \forall\ x\in E$.  Therefore $\eta_2=0$ a.e. in $E$ by the observation that we made. Hence $w_2^{+}(x)-w_2^{-}(x)=\eta_1(x)$  a.e. in $E$.
Clearly $w_2^{-}(x)\not>0$ a.e., otherwise $w_2^{+}(x)=0$ a.e. and that would imply $\eta_1(x)=-w_2^{-}(x)<0$ a.e, which is not possible since $\eta_1>0$ in $E$. Thus $w_2^{-}(x)=0$. Hence $\eta_1(x)=w_2^{+}(x)$ a.e. in $E$, which is a contradiction. Hence the claim follows.

\vspace{2mm}

Therefore $w_2$ is sign changing in $\Om$ and $u_n \deb w_2$ in $X_0$. 
Moreover, $I'_{\mu}(u_n)\to 0$ in $(X_0)'$ implies
\bea
\int_{\R^{2N}}\f{|u_n(x)-u_n(y)|^{p-2}(u_n(x)-u_n(y))(\phi(x)-\phi(y))}{|x-y|^{N+ps}}dxdy
&&-\mu\Iom|u_n|^{q-1}u_n\phi dx\no\\
&&-\Iom |u_n|^{p^*_s-2}u_n\phi dx=o(1)\no\\
\eea
for every $\phi\in X_0$.
Passing  the limit using Vitali's convergence theorem via H\"older's inequality we obtain $\<I'_\mu(w_2),\phi\>=0$.  Hence  $w_2$ is a sign changing weak solution to $(\mathcal{P}_{\mu})$.
\hfil{$\square$}

\begin{lemma}\label{l:6-i}
Let $u_{\eps,\de}$ be as defined in \eqref{u-eps} and $w_1$ be a positive solution of $(\mathcal{P}_{\mu})$ for which $\tilde\al_{\mu}^{-}$ is achieved, when $\mu\in (0,\mu_*)$. Then there exists $a, \ b \in \R, \ a \geq 0$ such that $aw_1-bu_\var \in \mathcal{N}^{-}_{*}$, where  $\mathcal{N}^{-}_{*}$ is defined as in \eqref{eq:N-star}.
\end{lemma}

This lemma can be proved in the spirit of \cite[Lemma 4.8]{BM}, for the convenience of the reader we again sketch the proof in the appendix.

\vspace{4mm}

{\bf Proof of Theorem \ref{thm.2}}: Define $\mu^*=\min\{\mu_*,\tilde\mu, \tilde\mu_1, \mu_0, \mu_1\}$, where $\mu_*$ is chosen such that $\tilde\al_{\mu}^-$ is achieved in $(0, \mu_*)$. $\tilde\mu$, $\tilde \mu_1$, $\mu_0$ and $\mu_1$ are as in \eqref{mu'}, Lemma \ref{inf-}, Theorem \ref{t:4i} and  Lemma \ref{l:6-ii} respectively.   
Furthermore, define $q_0$ and $N_0$ as follows:
\begin{align*}
q_{0}:= 
 \begin{cases}
  \text{max}\{q_1,q_2\} \quad\mbox{when}\quad p\geq \f{3+\sqrt{5}}{2},\\
  \text{max}\{q_1,q_3\}\quad\mbox{when} \quad 2\leq p<\f{3+\sqrt{5}}{2}.
 \end{cases}
\end{align*}
\begin{align*}
N_{0}:= 
 \begin{cases}
  sp(p^2-p+1) \quad\mbox{when}\quad p\geq \f{3+\sqrt{5}}{2},\\
  sp(p+1)\quad\mbox{when} \quad 2\leq p<\f{3+\sqrt{5}}{2}.
 \end{cases}
\end{align*}
Note that $N_0> \f{sp}{2} [p+1+\sqrt {(p+1)^2-4 }] $, where the RHS appeared in Theorem \ref{t:4ii}. Hence 
combining Theorem \ref{t:4ii} and Theorem \ref{t:4i}, we complete the proof of this theorem for $\mu\in(0,\mu^*)$, $q>q_0$ and $N>N_0$.
\hfil{$\square$}

\section{\bf Appendix} 

\begin{lemma}\label{l:6-ii}
Let $g_n$ be as in \eqref{g_n} in the Theorem \ref{t:4ii} and $v\in X_0$ such that $||v||_{X_0}=1$. Then there exists $\mu_1>0$ such that
if $\mu\in (0,\mu_1)$ implies $\<g_n'(0),v^+\>$ is uniformly bounded in $X_0$. 
\end{lemma}
\begin{proof}
 In view of lemma $\ref{N.mu-iii}$ we have,
$$\<g'_n(0),v^+\>=\displaystyle\frac{pA(u_n,v^+)-p^*_s\displaystyle\Iom |u_n|^{p^*_s-p}u_nv^+-(q+1)\mu\displaystyle\Iom |u_n|^{q-1}u_nv^+}
{(p-1-q)||{u_n}||_{X_0}^p-(p^*_s-q-1)|u_n|^{p^*_s}_{L^{p^*_s}(\Om)}}.$$
Using Claim 2 in theorem \ref{t:4ii}, there exists $C>0 $ such that $||{u_n}||_{X_0} \leq C$ for all $n \geq 1. $
Therefore applying H\"older inequality followed by  Sobolev inequality, we have\\
$|\<g'_n(0),v^+\>|\leq \frac{C ||{v}||_{X_0}}{\left|(p-1-q)||{u_n}||_{X_0}^p-(p^*_s-q-1)|u_n|^{p^*_s}_{L^{p^*_s}(\Om)}\right|}. $
Hence it is enough to show   
$$\left|(p-1-q)||{u_n}||_{X_0}^p-(p^*_s-q-1)|u_n|^{p^*_s}_{L^{p^*_s}(\Om)}\right|>C,$$ for some $C>0$ and $n$ large.
Suppose it does not hold. Then up to a subsequence 
$$(p-1-q)||{u_n}||_{X_0}^p-(p^*_s-q-1)|u_n|^{p^*_s}_{L^{p^*_s}(\Om)}=o(1)\quad\mbox{as}\quad n\to \infty. $$
Hence,
\begin{align} \label{*1}
||{u_n}||_{X_0}^p=\frac{p^*-q-1}{p-1-q}|u_n|^{p^*_s}_{L^{p^*_s}(\Om)}+o(1)\quad\mbox{as}\quad n\to \infty.
\end{align}
Combining  the above expression along with the fact that $u_n \in N_{\mu}$, we obtain
\begin{align} \label{*4}
\mu|u_n|^{q+1}_{L^{q+1}(\Om)}=\frac{p^*_s-p}{p-1-q}\abs{u_n}^{p^*_s}_{L^{p^*_s}(\Om)} +o(1)=\frac{p^*_s-p}{p^*_s-1-q}
||{u_n}||_{X_0}^p+o(1).
\end{align}
After applying H\"{o}lder inequality and followed by Sobolev inequality,  expression  \eqref{*4} yields 
\begin{align}\label{*6}
||{u_n}||_{X_0} \leq \bigg(\mu\frac{p^*_s-q-1}{p^*_s-p}|\Om|^{\frac{p^*_s-q-1}{p^*_s}}
S^{-\frac{q+1}{p}} \bigg)^{\frac{1}{p-1-q}} +o(1).
\end{align}
Combining \eqref{4-13} and Claim 3 in the proof of Theorem \ref{t:4ii}, we have $||{u_n}||_{X_0}\geq b$, 
for some $b>0$. Therefore from \eqref{*1} we get
\begin{align}\label{*9}
|u_n|^{p^*_s}_{L^{p^*_s}(\Om)} \geq C\quad\mbox{for some constant}\  C>0,\ \mbox{and}\,\, n\,\,\mbox{large enough.}
\end{align}
Define $\psi_\mu: N_{\mu} \to \R$ as follows:
$$\psi_\mu(u)=k_0\bigg(\frac{||{u}||_{X_0}^{p(p^*_s-1)}}{|u|^{p^*_s(p-1)}_{L^{p^*_s}(\Om)}}\bigg)
^{\frac{1}{p^*_s-p}}-\mu|u|^{q+1}_{L^{q+1}(\Om)},$$
where $k_0=\left(\frac{p-1-q}{p^*_s-q-1}\right)^{\frac{p^*_s-1}{p^*_s-p}}\left(\frac{p^*_s-p}{p-1-q}\right)$. 
Simplifying  $\psi_{\mu}(u_n)$ using \eqref{*4}, we obtain
\be \label{*7}
\psi_\mu(u_n)=k_0\bigg[\bigg(\f{p^*_s-q-1}{p-1-q}\bigg)^{p^*_s-1}\f{|u_n|^{(p^*_s-1)p^*_s}_{L^{p^*_s}(\Om)}}
{|u_n|_{L^{p^*_s}(\Om)}^{p^*_s(p-1)}}\bigg]^\f{1}{p^*_s-p} - \f{p^*_s-p}{p-1-q} |u_n|_{L^{p^*_s}(\Om)}^{p^*_s}+o(1)=o(1).
\ee
On the other hand, using H\"older inequality in the definition of $\psi_{\mu}(u_n)$, we obtain
\bea\lab{10*}
 \psi_\mu(u_n)&=&k_0\bigg(\frac{||{u_n}||_{X_0}^{p(p^*_s-1)}}{|u_n|^{p^*_s(p-1)}_{L^{p^*_s}(\Om)}}\bigg)^{\frac{1}{p^*_s-p}}-
 \mu|u_n|^{q+1}_{L^{q+1}(\Om)}\notag \\               
&\geq&k_0\bigg(\frac{||{u_n}||_{X_0}^{p(p^*_s-1)}}{|u_n|^{p^*_s(p-1)}_{L^{p^*_s}(\Om)}}\bigg)^{\frac{1}{p^*_s-p}}-
\mu|\Om|^{\f{p^*_s-q-1}{p^*_s}}|u_n|^{q+1}_{L^{p^*_s}(\Om)}\notag\\
&=&|u_n|^{q+1}_{L^{p^*_s}(\Om)}\bigg\{k_0\bigg(\frac{||{u_n}||_{X_0}^{p(p^*_s-1)}}{|u_n|^{p^*_s(p-1)}_{L^{p^*_s}(\Om)}}\bigg)
^{\frac{1}{p^*_s-p}}\frac{1}{|u_n|^{q+1}_{L^{p^*_s}(\Om)}}-\mu|\Om|^{\f{p^*_s-q-1}{p^*_s}}\bigg\}.
\eea
Using Sobolev embedding and \eqref{*6}, we simplify the term 
$\bigg(\frac{||{u_n}||_{X_0}^{p(p^*_s-1)}}{|u_n|^{p^*_s(p-1)}_{L^{p^*_s}(\Om)}}\bigg)
^{\frac{1}{p^*_s-p}}\frac{1}{|u_n|^{q+1}_{L^{p^*_s}(\Om)}}$ and obtain
\bea\lab{11*}
\bigg(\frac{||{u_n}||_{X_0}^{p(p^*_s-1)}}{|u_n|^{p^*_s(p-1)}_{L^{p^*_s}(\Om)}}\bigg)
^{\frac{1}{p^*_s-p}}\frac{1}{|u_n|^{q+1}_{L^{p^*_s}(\Om)}} &\geq& S^\f{p^*_s-1}{p^*_s-p}|u_n|_{L^{p^*_s}(\Om)}^{-q}\no\\
&\geq& S^{\f{p^*_s-1}{p^*_s-p}+\f{q}{p}}||u_n||_{X_0}^{-q}\no\\
&\geq& S^{\f{p^*_s-1}{p^*_s-p}+\f{q}{p}} \bigg(\mu\frac{p^*_s-q-1}{p^*_s-p}|\Om|^{\frac{p^*-q-1}{p^*_s}}
S^{-\frac{q+1}{p}} \bigg)^{-\frac{q}{p-1-q}}.
\eea
Substituting back \eqref{11*} into \eqref{10*} and using \eqref{*9}, we obtain
$$
 \psi_\mu(u_n)\geq C^{q+1}\bigg [k_0 S^{\f{p^*_s-1}{p^*_s-p}+\f{q}{p-1-q}}
 \mu^{-\f{q}{p-1-q}}\left(\frac{p^*_s-q-1}{p^*_s-p}|\Om|^{\frac{p^*_s-q-1}{p^*_s}}\right)^{-\f{q}{p-1-q}}
 -\mu|\Om|^\f{p^*_s-q-1}{p^*_s}\bigg]\geq d_0,
$$
for some $d_0>0$, $n$ large and $\mu < \mu_1$, where
$\mu_1=\mu_1(k_0, s, q, N, |\Om|)$. This is a contradiction to \eqref{*7}. Hence the lemma follows. 
\end{proof}

{\bf Proof of Lemma \ref{l:6-i}} \begin{proof}
 We will show that there exists $a>0, \ b \in \R$ such that 
 $$a(w_1-bu_\var)^{+} \in N_{\mu}^{-} \quad\mbox{and}\quad -a(w_1-bu_\var)^{-} \in N_{\mu}^{-}. $$
 Let us denote $\bar r_1=\inf_{x \in \Om}\frac{w_1(x)}{u_\var(x)}, \,\, \bar r_2=\sup_{x \in \Om}\frac{w_1(x)}{u_\var(x)}$.\\
 As both $w_1$ and $u_\var$ are positive in $\Om$, we have $\bar r_1 \geq 0$ and  $\bar r_2$ can be $+\infty$.
Let $r \in (\bar r_1,\bar r_2)$. Then  $w_1,u_\var \in X_0$ implies $(w_1-ru_\var) \in X_0$ and $(w_1-ru_\var)^{+} \nequiv 0$.
Otherwise, $(w_1-ru_\var)^{+} \equiv 0$ would imply $\bar r_2 \leq r$, which is not possible. Define $v_r := w_1-ru_\var$. 
Hence $0\not\equiv v_r^{+}\in X_0$ (since for any $u\in X_0$, we have $|u|\in X_0$. Similarly $0\not\equiv v_r^{-}\in X_0$. Therefore by lemma $\ref{N.mu-i}$ there exists $0<s^{+}(r)<s^{-}(r)$ such that $s^{+}(r)v^{+}_r \in N_{\mu}^{-}$, 
 and $-s^{-}(r)(v_r^{-}) \in N_{\mu}^{-} $.
Let us consider the functions $s^{\pm}: \R \to (0,\infty)$  defined as above.\\
\textit{Claim}: The functions $r\mapsto s^{\pm}(r)$ are continuous and $$\lim_{r \to \bar r_1^+}s^{+}(r)=t^{+}(v^{+}_{\bar r_1}) \quad\text{and}\quad
\lim_{r \to \bar r_2^-}s^{+}(r)=+\infty,$$ where the function $t^{+}$ is same as defined  in lemma \ref{N.mu-i}.\\
To see the claim, choose $r_0 \in (\bar r_1,\bar r_2)$ and $\{r_n\}_{n \geq 1} \subset (\bar r_1,\bar r_2)$  such that $r_n \to r_0$ as $n \to \infty$. We need to show 
that $s^{+}(r_n)\to s^{+}(r_0)$ as $n \to \infty$.  Corresponding to $r_n$ and $r_0$, we have  $v_{r_n}^{+}=(w_1-r_nu_\var)^{+}$ and $v_{r_0}^{+}=(w_1-r_0u_\var)^{+}$.
By lemma $\ref{N.mu-i}$. we note that $s^{+}(r)=t^{+}(v^{+}_r)$.
Let us define the function 
\Bea
F(s, r)& :=&s^{p-1-q}||{(w_1-ru_{\eps})^+}||_{X_0}^p-s^{p^*_s-q-1}|(w_1-ru_{\eps})^+|^{p^*_s}_{L^{p^*_s}(\Om)}
-\mu|(w_1-ru_{\eps})^+|^{q+1}_{L^{q+1}(\Om)}\\
&=& \phi(s, r)-\mu|(w_1-ru_{\eps})^+|^{q+1}_{L^{q+1}(\Om)},
\Eea
where $$\phi(s,r) := s^{p-1-q}||{(w_1-ru_{\eps})^+}||_{X_0}^p-s^{p^*_s-q-1}|(w_1-ru_{\eps})^+|^{p^*_s}_{L^{p^*_s}(\Om)}.$$ 
Doing the similar calculation as in lemma $\ref{N.mu-i}$, we obtain that for any fixed $r$, the function $F(s, r)$ has only two zeros $s=t^{+}(v_r^+)$ and $s=t^{-}(v_r^+)$. Consequently $s^{+}(r)$ is the largest $0$ of $F(s, r)$ for  any fixed $r$. As $r_n \to r_0$ we have $v^{+}_{r_n} \to v^{+}_{r_0}$ in $X_0$ .
Indeed, by straight forward computation it follows $v_{r_n}\to v_{r_0}$ in $X_0$. Therefore $|v_{r_n}|\to |v_{r_0}|$ in $X_0$. This in turn implies $v_{r_n}^+\to v_{r_0}^+$ in $X_0$.
Hence $||v_{r_n}^+||_{X_0}\to ||v_{r_0}^+||_{X_0}$. Moreover by Sobolev inequality, we have $|v_{r_n}^+|_{L^{p^*_s}(\Om)}\to 
|v_{r_0}^+|_{L^{p^*_s}(\Om)}$ and  $|v_{r_n}^+|_{L^{q+1}(\Om)}\to |v_{r_0}^+|_{L^{q+1}(\Om)} $. As a result, we have
$F(s, r_n)\to F(s, r_0)$ uniformly. Therefore an elementary analysis yields $s^+(r_n)\to s^+(r_0)$.

Moreover, $\bar r_2\geq \f{w_1}{u_\eps}$ implies $w_1-\bar r_2u_{\eps}\leq 0$.
As a consequence $r\to \bar r_2^-$ implies $(w_1-r u_{\eps})^+\to 0$ pointwise.
Moreover, since $|(w_1-r u_{\eps})^+|_{L^{\infty}(\Om)}\leq |w_1|_{L^{\infty}(\Om)}$,
using dominated convergence theorem we have $|(w_1-r u_{\eps})^+|_{L^{p^*_s}(\Om)}\to 0$.
From the analysis in Lemma \ref{N.mu-i}, for any $r$, we also have $s^+(r)>t_0(v_r^+)$, where function $t_0$ is defined as in lemma \ref{N.mu-i}, which is the maximum point of $\phi(.,r)$.
Therefore it is enough to show that $\lim_{r\to \bar r_2^-}t_0(v_r^+)=\infty$. Applying Sobolev inequality in the definition of $t_0(v_r^+)$ we get
$$t_0(v_r^+)=\bigg(\f{(p-1-q)||v_r^+||_{X_0}^p}{(p^*_s-1-q)|v_r^+|_{L^{p^*_s}(\Om)}^{p^*_s}}\bigg)^\f{1}{p^*_s-p}
\geq\bigg(\f{S(p-1-q)}{p^*_s-1-q}\bigg)^\f{1}{p^*_s-p}|v_r^+|_{L^{p^*_s}(\Om)}^{-1}.$$ 
Hence $\lim_{r\to \bar r_2^-}t_0(v_r^+)=\infty$.  

Proceeding similarly we can show that if $r\to \bar r_1^-$ then $v_r^+\to v_{\bar r_1}$ and 
$\lim_{r \to \bar r_1^+}s^{+}(r)=t^{+}(v^{+}_{\bar r_1})$ and
$$\lim_{r \to r_1^{+}}s^{-}(r)=+\infty,\,\, \lim_{r\to r_2^{-}}s^{-}(r)=t^{+}(v_r^-)<+\infty. $$
The continuity of $s^{\pm}$ implies that there exists $b \in (\bar r_1,\bar r_2)$ such that
$s^{+}(r)=s^{-}(r)=a>0$.
Therefore, $$a(w_1-bu_\var)^{+} \in N_{\mu}^{-} \quad\mbox{and}\quad -a(w_1-bu_\var)^{-} \in N_{\mu}^{-},$$
that is, the function $a(w_1-bu_\var) \in \mathcal{N}_*^{-}$ and this completes the proof.

\end{proof}

\vspace{5mm}

{\bf Acknowledgement:}  The first author is supported by the INSPIRE research grant DST/INSPIRE 04/2013/000152 and the second author is supported by the
NBHM grant 2/39(12)/2014/RD-II.

\end{document}